\documentclass{article}
\usepackage{amsmath, amssymb}

\newtheorem{thm}{Theorem}[section]
\newtheorem{cor}[thm]{Corollary}
\newenvironment{proof}[1][,]{\medskip\ifcat,#1
\noindent{{\it Proof\/}:\ }\else\noindent{\it Proof of #1.\ }\fi}
{\hfill$\square$\medskip}
\newenvironment{dfn}{\medskip\refstepcounter{thm}
\noindent{\bf Definition \thesection.\arabic{thm}\ }}{\medskip}
\newtheorem{prop}[thm]{Proposition}
\newenvironment{ex}{\medskip\refstepcounter{thm}
\noindent{\bf Example \thesection.\arabic{thm}\ }}{\medskip}

\newcounter{condition}
\newenvironment{note}[1][Note]{\begin{trivlist}
\item[\hskip \labelsep {\bfseries #1}]}{\end{trivlist}}
\newenvironment{conds*}[1][Conditions]{\begin{trivlist}
\item[\hskip \labelsep {\bfseries #1}]}{\end{trivlist}}
\newenvironment{remark}[1][Remark]{\begin{trivlist}
\item[\hskip \labelsep {\bfseries #1}]}{\end{trivlist}}
\newtheorem{lem}[thm]{Lemma}
\newenvironment{remarks}[1][Remarks]{\begin{trivlist}
\item[\hskip \labelsep {\bfseries #1}]}{\end{trivlist}}
\newenvironment{prob}[1][Problem]{\begin{trivlist}
\item[\hskip \labelsep {\bfseries #1}]}{\end{trivlist}}
\newenvironment{ack}[1][Acknowledgements]{\begin{trivlist}
\item[\hskip \labelsep {\bfseries #1}]}{\end{trivlist}}

\def\eq#1{{\rm(\ref{#1})}}

\def\d{{\rm d}}
\def\R{\mathbb{R}}
\def\w{\wedge}
\DeclareMathOperator\supp{supp}
\DeclareMathOperator\Span{Span}
\def\Z{\mathbb{Z}}
\DeclareMathOperator\vol{vol}
\DeclareMathOperator\Imm{Im}
\DeclareMathOperator\Ker{Ker}
\DeclareMathOperator\GG{G}
\def\N{\mathbb{N}}
\DeclareMathOperator\SU{SU}
\def\P{\mathbb{P}}
\def\C{\mathbb{C}}
\DeclareMathOperator\Stab{Stab}
\DeclareMathOperator\GL{GL}
\DeclareMathOperator\SO{SO}
\DeclareMathOperator\Hol{Hol}
\DeclareMathOperator\id{id}
\DeclareMathOperator\ind{ind}
\DeclareMathOperator\AAA{A}

\begin{document}

\title{Desingularization of Coassociative 4-folds with Conical Singularities: Obstructions and Applications}
\author{\textsc{Jason D. Lotay\footnote{The author is supported by an EPSRC Career Acceleration Fellowship.}
}\\ University College London}
\date{}

\maketitle

\begin{abstract}
\noindent  We study the problem of desingularizing coassociative conical singularities via gluing, allowing for topological and 
analytic obstructions, and discuss applications.  This extends work in \cite{Lotaydesing} on the unobstructed case.  
We interpret the analytic obstructions geometrically via the obstruction theory for deformations of conically singular coassociative 4-folds, and thus
relate them  
to the stability of the 
singularities.  We use our results to describe the relationship between  moduli 
spaces of coassociative 4-folds with conical singularities and those of their desingularizations.  We also apply our theory in  examples, including to the known conically singular coassociative 4-folds  in compact holonomy $\GG_2$ manifolds.
\end{abstract}

\section{Introduction}

Coassociative 4-folds are calibrated 4-dimensional submanifolds in 7-manifolds with exceptional holonomy $\GG_2$ (and, more generally, in 7-manifolds with a $\GG_2$ structure).
Studying gluing problems for calibrated submanifolds in manifolds with special holonomy has proven to be a rich and fruitful avenue of research, 
 particularly in special Lagrangian geometry in the work of Joyce \cite{JoyceCS3, JoyceCS4}, Haskins and Kapouleas \cite{HaskinsKapouleas} and 
Pacini \cite{PaciniSL2}, as well as for associative \cite{Nordstrom} and coassociative geometry 
\cite{Lotaydesing}.  In particular, the desingularization problem for calibrated submanifolds with 
conical singularities naturally feeds into the understanding of the 
boundary of the moduli space of smooth calibrated submanifolds, and has crucial consequences for the construction of potential invariants for manifolds with special 
holonomy by suitable ``counting'' of calibrated submanifolds (see, for example, \cite{Joycecount, JoyceCS5} for a discussion of these issues in the special Lagrangian case).

In any gluing problem one naturally has to tackle the issue of \emph{obstructions}.  This is normally achieved by making strong assumptions on the geometry of the 
submanifolds to be glued, for example in \cite{JoyceCS3, Lotaydesing, PaciniSL2}, or by restricting to situations where the obstructions can 
be identified and resolved in a natural way, either via topological conditions (as in \cite{JoyceCS4}) or symmetries of the problem (as in \cite{HaskinsKapouleas}).  In this article we extend the work in \cite{Lotaydesing} and consider a gluing problem in coassociative geometry where we deal with both topological and analytic obstructions in desingularizing isolated conical singularities.  We interpret the analytic obstructions geometrically using the deformation theories of asymptotically conical and
conically singular coassociative 4-folds developed in \cite{LotayCS, LotayAC}.  We thus relate our obstructions to the notion of \emph{stability} of coassociative 
conical singularities introduced in \cite{Lotaystab}.  

We use our desingularization results to help describe how the moduli spaces of asymptotically conical, conically singular and smooth compact coassociative 4-folds are related, and thus provide a greater understanding of the boundary of the 
moduli space of compact coassociative 4-folds.  In the case of stable conical singularities, this enables us to construct a local diffeomorphism 
 between the gluing data associated to the singular coassociative 4-fold and a neighbourhood of the coassociative smoothing ``near the boundary'' of the moduli 
 space.  

We also discuss examples where our desingularization theory applies, including the first known examples of coassociative 4-folds with conical singularities in 
compact manifolds with $\GG_2$ holonomy, which were constructed in \cite{Lotaystab}.
 
\medskip

The setting in this article is the following.  We have a coassociative 4-fold $N$ in an almost $\GG_2$ manifold $M$ (a 7-manifold with a closed $\GG_2$ structure) and we suppose that $N$ 
has a single conical singularity $z$ modelled on a cone $C$.  We also assume there exists a coassociative 4-fold $A$ in $\R^7$ which is asymptotically conical with rate $\lambda<-\frac{1}{2}$ to 
the cone $C$.  (See $\S$2 for precise definitions.)  

\medskip

In Definition \ref{matchingdfn} we give a \emph{matching condition} between $A$ and $N$.   The matching condition is a mixture of topological and analytic constraints, which then allows us to deal with both topological and analytic obstructions.  

 To give a sense of the matching condition we make some observations.
 Let $\Sigma$ be the link of $C$ and let $j^A_2:H^2(A)\rightarrow H^2(\Sigma)$ and $j_2^N:H^2(\hat{N})\rightarrow H^2(\Sigma)$ be the induced maps arising from 
inclusion of $\Sigma$ in $A$ and $\hat{N}=N\setminus\{z\}$ (the non-compact manifold given by removing the singularity from $N$). 

If $\varphi_0$ is the standard $\GG_2$ structure on $\R^7$, then since $\varphi_0$ is closed and $\varphi_0|_A=0$ we have an element $[\varphi_0]\in H^3(\R^7,A)\cong H^2(A)$, which we may also 
view as the cohomology class of the infinitesimal deformation of $A$ corresponding to dilation.  The natural topological constraint is therefore that $j_2^A[\varphi_0]$ lies in $\Imm j_2^N$.  

We can relate the analytic obstructions to the obstruction theory of $N$ and thus to the notion of \emph{$\mathcal{C}$-stability} of the cone $C$ (see Definition 
\ref{stabdfn}) for a deformation family $\mathcal{C}$ of $C$   -- this condition is discussed in detail in \cite{Lotaystab}.  

Overall, we have the following interpretation of the matching condition.  

\begin{prop}\label{mainthm3}
If $j_2^A[\varphi_0]\in \Imm j_2^N$ and the cone $C$ at the singularity is $\mathcal{C}$-stable, then the matching condition is satisfied.
\end{prop}

We show that the matching condition allows us to desingularize $N$ using $A$ via gluing, giving our main result.

\begin{thm}\label{mainthm1}
If the matching condition is satisfied, there exists $\tau>0$ such that for all $t\in(0,\tau)$ there is a smooth 
compact coassociative 4-fold $N(t)$ in $M$, formed by gluing $tA$ and $N$, such that $N(t)$ converges to $N$ in the sense of currents as $t\rightarrow 0$.
\end{thm}

By deforming $N(t)$ we obtain a family of desingularizations of $N$ whose dimension we can determine from the topology of $A$ and $N$.  Recall 
that for a noncompact 4-manifold we define $b^2_+$ by considering the cup product on cohomology classes representing by compactly supported 2-forms.

\begin{cor}\label{mainthm2}
If the matching condition is satisfied, we have a smooth family of nearby compact coassociative smoothings of $N$ of dimension $b^2_+(A)+b^2_+(\hat{N})+\dim(\Imm j_2^A\cap\Imm j_2^N)$.
\end{cor}

This shows that the moduli space of smooth compact coassociative 4-folds can be non-compact and that 
coassociative 4-folds with conical singularities can arise on the boundary of the moduli space.  Moreover, we have
 a \emph{gluing map} 
from the moduli space of matching pairs $(N,tA)$ into the moduli space of smooth compact coassociative 4-folds, from which we may deduce the following.  

\begin{prop}\label{mainthm4}
If $C$ is stable, the gluing map is a local diffeomorphism. 
\end{prop}

The organisation of the paper is as follows.  
\begin{itemize}
\item In $\S$2 we provide the basic definitions and notation which shall be used throughout the paper and discuss foundational results for coassociative 4-folds 
involving self-dual 2-forms and tubular neighbourhood theorems.  
\item In $\S$3 we describe our gluing construction, identify the obstructions to the procedure 
and the necessary matching condition for the construction to succeed.  We also set up the analytic framework for our problem and discuss the relationship between our matching conditions and the deformation theory of $A$ and $N$, which allows us to prove Proposition \ref{mainthm3}.
\item In $\S$4, we realise 
our smoothing of the singular coassociative 4-fold as the fixed point of a map between Banach spaces, which we prove is a contraction by deriving appropriate analytic estimates on the smoothing using 
estimates on the ``building blocks'' $A$ and $N$.  We deduce Theorem \ref{mainthm1} and Corollary \ref{mainthm2} from this work.
\item In $\S$5 we compare the moduli space of ``matching pairs'' $(N,tA)$ to the moduli space of smoothings and deduce Proposition \ref{mainthm4}.  
We conclude by applying our theory in examples.
\end{itemize}


\section{Foundations}

In this section we describe the basic foundational material we need to tackle our desingularization problem.

\subsection{Basic definitions}

There exists a 3-form $\varphi_0$ on $\R^7$ with constant coefficients 
such that $\Stab(\varphi_0)\subseteq\GL(7,\R)$ is isomorphic to $\GG_2$.  
  In fact $\Stab(\varphi_0)\subseteq \SO(7)$ so $\GG_2$ preserves the Euclidean metric $g_0$ and orientation on $\R^7$.
  
\begin{dfn}
We call a 3-form $\varphi$ on an oriented 7-manifold $M$ a \emph{$G_2$ structure} if, for all $x\in M$, 
$\varphi|_x=\iota_x^*(\varphi_0)$ for some orientation preserving isomorphism $\iota_x:T_xM\rightarrow\R^7$.  
A $\GG_2$ structure $\varphi$ defines a metric $g_{\varphi}$ on $M$.       
\end{dfn}

We denote a 7-manifold $M$ endowed with a $\GG_2$ structure $\varphi$ by $(M,\varphi)$.  
An oriented 7-manifold will admit a $\GG_2$ structure if and only if it is spin.  We now define special classes of $\GG_2$ structures which will be especially 
relevant.

\begin{dfn} 
We say that $(M,\varphi)$ is an \emph{almost $G_2$ manifold} if $\d\varphi=0$.  We call $(M,\varphi)$ a \emph{$G_2$ manifold} if
 $\d\varphi=\d^*\varphi=0$ with respect to $g_{\varphi}$, which is equivalent to saying that the holonomy  $\Hol(g_{\varphi})\subseteq\GG_2$. 
\end{dfn}

In 7-manifolds with a $\GG_2$ structure we have a distinguished class of submanifolds which shall form the basis for our study.

\begin{dfn}\label{coassdfn}
A \emph{coassociative 4-fold} $X$ in $(M,\varphi)$ is a 4-dimensional submanifold of $M$ such that 
$\varphi|_X\equiv 0$, oriented so that $*\varphi|_X>0$.  Equivalently, $X$ is coassociative if and only if $*\varphi|_X=\vol_X$.
\end{dfn}

\noindent When $(M,\varphi)$ is a $\GG_2$ manifold, coassociative 4-folds are volume-minimizing in their homology class.  Although coassociative 4-folds
lose this property in general almost $\GG_2$ manifolds, their geometry otherwise has essentially the same features as in the $\GG_2$ manifold case.  
Coassociative geometry is discussed in detail in \cite{JoyceRiem}. 

Let $B(0;r)$ denote the Euclidean ball of radius $r$ about $0$ in $\R^7$.  For a cone $C$ in $\R^7$ (i.e.~a dilation-invariant subset) such that $C\setminus\{0\}$ is a smooth submanifold we let $\Sigma=C\cap\mathcal{S}^6$ with the induced metric $g_{\Sigma}$, 
let $\iota:C\cong\R^+\times\Sigma\rightarrow\R^7$ be the inclusion map and let $\nabla_C$ be the Levi-Civita connection of the cone metric $g_C=\d r^2+r^2g_{\Sigma}$ on $C$.  

We may now define the two types of submanifold which shall appear in our desingularization problem: namely \emph{conically singular} and \emph{asymptotically conical}.  In 
each case we have a noncompact submanifold of $(M,\varphi)$ which, outside a compact set, is diffeomorphic to a cone  (or finite collection of cones) 
and converges to the cone with some prescribed rate.  The two classes of submanifold will be dual in the sense that one converges towards the cone near its vertex (so has a singular point) whereas 
the other converges to the cone near infinity.

\begin{dfn}\label{csdfn}
Let $N$ be a (singular) submanifold of $(M,\varphi)$ and let $z\in N$.  Choose local coordinates $\chi:B(0;\epsilon_M)\rightarrow V\ni z$, for some $\epsilon_M\in(0,1)$ 
and open $V\subseteq M$, such that $\chi(0)=z$ and $(\d\chi|_0)^*(\varphi|_z,g_\varphi|_z)=(\varphi_0,g_0)$.  (These are natural coordinates for a neighbourhood of $z$ in $(M,\varphi)$.)

We say that $N$ has a \emph{conical singularity} at $z$ if there exist a cone $C\subseteq\R^7$ with 
link $\Sigma\subseteq\mathcal{S}^6$,  constants $\epsilon\in(0,\epsilon_M)$ and $\mu\in(1,2)$, open $U\subseteq V\cap N$ containing $z$ and a smooth map
$$\Phi_N:(0,\epsilon)\times\Sigma\rightarrow B(0;\epsilon_M)\quad\text{such that}\quad\Psi_{N}=\chi\circ\Phi_N:(0,\epsilon)\times\Sigma\rightarrow U\setminus\{z\}$$ is a diffeomorphism and 
\begin{equation}\label{Nasymeq}
\big|\nabla_C^j\big(\Phi_N(r,\sigma)-\iota(r,\sigma)\big)\big|
=O\big(r^{\mu-j}\big)\qquad \text{for $j\in\N$ as $r\rightarrow 0$ on $C$.}
\end{equation}
We call $C$ the \emph{cone} and $\mu$ the \emph{rate} at the singularity.  

We say that $N$ is a \emph{conically singular} (CS) submanifold 
if $N$ is compact and connected and smooth except for finitely many conical singularities.  We call $N$ a \emph{CS coassociative 4-fold} if $N$ is a 
CS submanifold whose nonsingular part is a coassociative 4-fold.
\end{dfn}

\begin{remarks}\begin{itemize}\item[]
\item[(a)]By \cite[Proposition 3.6]{LotayCS}, if $N$ is a CS coassociative 4-fold then the cones at the singularities are coassociative in $\R^7$.  
\item[(b)] The stipulation that $\mu<2$ allows the definition of conical singularity to be essentially independent of 
the choice of local coordinates $\chi$, as explained in \cite[$\S$3.2]{LotayCS}.  
\item[(c)] Notice that if $N$ is CS with rate $\mu_0$ it is also CS with any rate $\mu\in(1,\mu_0]$.  We are thus
free to reduce the rate $\mu$, so we always choose $\mu$ close to $1$.
\end{itemize}
\end{remarks}
 
\begin{dfn}\label{acdfn}
A (smooth) submanifold $A$ of $\R^7$ is an \emph{asymptotically conical} (AC) submanifold 
if there exist a cone $C\subseteq\R^7$ with link $\Sigma\subseteq\mathcal{S}^6$, constants $R>0$ and $\lambda<1$, 
compact $K_A\subseteq A$ and a diffeomorphism 
$\Phi_A:(R,\infty)\times\Sigma\rightarrow A\setminus K_A$ satisfying 
\begin{equation}\label{Aasymeq}
\big|\nabla^j_C\big(\Phi_A(r,\sigma)-\iota(r,\sigma)\big)\big|=O\big(r^{\lambda-j}\big)
\qquad\text{for $j\in\N$ as $r\rightarrow\infty$ on $C$.}
\end{equation}
We say that $A$ is AC with \emph{rate} $\lambda$ to $C$ to emphasise the choice of $C$ and $\lambda$.
\end{dfn}

\begin{remarks}
\begin{itemize}\item[]
\item[(a)] By \cite[Proposition 2.8]{LotayAC}, if $A\subseteq\R^7$ is coassociative and AC to $C$ then $C$ is coassociative.
\item[(b)] Observe that $A$ only genuinely converges to $C$ at infinity if $\lambda<0$, so allowing for $\lambda\in[0,1)$ permits weak decay.
\item[(c)] Note that if $A$ is AC with rate $\lambda_0$ it is also AC with any higher rate $\lambda\in[\lambda_0,1)$, 
so we are at liberty to increase the rate $\lambda$.
\end{itemize}
\end{remarks}

As we see, AC submanifolds are smoothings of cones and thus provide obvious models for desingularizing CS submanifolds via gluing.  However, the challenge is to 
 desingularize CS \emph{coassociative} 4-folds so that the coassociative condition is preserved, so one would naturally require AC coassociative 4-folds
 in the gluing.   
Our problem is to study when this approach may be successfully applied and
 understand the obstructions to the coassociative gluing process.

To fix notation we describe the ingredients we wish to feed into our problem.

\begin{itemize}
\item Let $C$ be a coassociative cone in $(\R^7,\varphi_0)$ with link $\Sigma=C\cap\mathcal{S}^6$. 
\item Let $N$ be a CS coassociative 4-fold in an almost $\GG_2$ manifold $(M,\varphi)$ with a single conical 
singularity at $z$ with rate $\mu$ and cone $C$.

Choose local coordinates $\chi:B(0;\epsilon_M)\rightarrow M$ about $z$ as in Definition \ref{csdfn}. 
 By Definition \ref{csdfn}, there exist $\epsilon\in(0,\epsilon_M)$, compact $K_N\subseteq \hat{N}=N\setminus\{z\}$ 
and a smooth map 
$$\Phi_N:(0,\epsilon)\times\Sigma\rightarrow B(0;\epsilon_M)\quad\text{such that}\quad \Psi_{N}=\chi\circ\Phi_N:(0,\epsilon)\times\Sigma\rightarrow \hat{N}\setminus K_N$$ is a diffeomorphism satisfying \eq{Nasymeq}.
We can choose $\Phi_N$ such that $$\Phi_N(r,\sigma)-\iota(r,\sigma)\in (T_{r\sigma}C)^{\perp}.$$
\item 
Let $A\subseteq\R^7$ be a coassociative 4-fold which is AC with rate $\lambda<1$ to $C$.  

By Definition \ref{acdfn}, there exist $R>0$, compact $K_A\subseteq A$, and a diffeomorphism
$\Phi_A:(R,\infty)\times\Sigma\rightarrow A\setminus K_A$ satisfying \eq{Aasymeq}.  We can also choose $\Phi_A$ so that 
$$\Phi_A(r,\sigma)-\iota(r,\sigma)\in (T_{r\sigma}C)^{\perp}.$$  Whenever $t$ is sufficiently small that $t^{-1}\epsilon>R$, we set $$\hat{A}(t)=K_A\cup \Phi_A\big((R,t^{-1}\epsilon)\times\Sigma\big).$$
 The subsets $t\hat{A}(t)\subseteq tA$ will be glued to $\hat{N}$ to resolve the singularity $z$.
\item 
Let $\tau\in(0,1)$ be small enough that $\tau R<\epsilon$ and $\tau \hat{A}(\tau)\subseteq B(0;\epsilon_M)$.  Throughout we let 
$t\in(0,\tau)$ be arbitrary and will make $\tau$ smaller a finite number of times, continuing to refer to this new smaller 
constant as $\tau$.  Our constraints on $\tau$ ensure that we can use the local coordinates $\chi$ to 
view the gluing of $t\hat{A}(t)$ to $\hat{N}$, outside the compact set $K_N$, as 
occurring in $B(0;\epsilon_M)\subseteq\R^7$.
\end{itemize}

\noindent We shall occasionally refer to $\hat{N}\setminus K_N$ and $A\setminus K_A$ as the end (or ends since they could be 
disconnected) of $\hat{N}$ and $A$.

\begin{remarks}
\begin{itemize}\item[]
\item[(a)]   We assume that $N$ is connected but not that $\hat{N}$ is connected.  In the special Lagrangian desingularization problem (as in 
\cite{JoyceCS3, JoyceCS4}) there is a marked difference when the corresponding $\hat{N}$ is connected or not.  This is because a
 topological obstruction in the special Lagrangian case automatically vanishes if $\hat{N}$ is connected.  For coassociative 4-folds, assuming that 
 $\hat{N}$ is connected does not force any such vanishing of topological obstructions, so such a difference when $\hat{N}$ is connected or not
does not occur.  The connectedness issue for $\hat{N}$ is crucial for special Lagrangians because a transverse intersection point is a simple example of a 
conical singularity and occurs naturally, since special Lagrangians are half the dimension of the ambient space, so removing the intersection point can 
clearly result in a disconnected $\hat{N}$.  However, coassociative 4-folds do not intersect transversely, but rather in curves generically, so the connectedness issue for $\hat{N}$ is not especially relevant in this geometry.
\item[(b)]  We assume that $N$ has a single singular point only 
for convenience and to avoid proliferation of notation since the same methods detailed in this article will be applicable to the case of 
multiple singular points.  
\end{itemize}
\end{remarks}

\subsection{Self-dual 2-forms and tubular neighbourhoods}

In geometric gluing problems it is often crucial to know the relationship between deformations of the building blocks and those of the glued object.
We shall therefore need to understand deformations of coassociative 4-folds, for which the key result
 is the following \cite[c.f.~Proposition 4.2]{McLean}.

\begin{prop}\label{jmathprop}
Let $X$ be a coassociative 4-fold in $(M,\varphi)$. There is an
isometric isomorphism $\jmath_X$ between the normal bundle $\nu(X)$ of $X$ in $M$ and
$\Lambda^2_+T^*X$ given by $v\mapsto(v\lrcorner\varphi)|_{X}$.
\end{prop}

\begin{note}
For any coassociative 4-fold $X$ we will consistently use the notation $\jmath_X$ to indicate the isomorphism in Proposition \ref{jmathprop}.
\end{note}

Using this identification, we can view nearby submanifolds to $X$ as graphs of small self-dual 2-forms; that
 is, give open neighbourhoods of the zero section in $\Lambda^2_+T^*X$ and of $X$ in $M$ and a diffeomorphism between them which 
acts as the identity on $X$ (identified with the zero section as usual).  However, we must perform this construction carefully so as to
 ensure compatibility with $\jmath_X$ and to take into account the asymptotic behaviour of $A$ and $N$.  We first make the
 compatibility property precise.
  
\begin{dfn}  Let $X$ be a coassociative 4-fold in $(M,\varphi)$.  
Suppose we have a smooth map $\Upsilon_X$ from an open neighbourhood of the zero section in $\Lambda^2_+T^*X$ to an open tubular
 neighbourhood of $X$ in $M$, acting as the identity $\id_X$ on $X$.  
We may then view $\d\Upsilon_X|_X$ as a map from $TX\oplus\Lambda^2_+T^*X$ to $TX\oplus\nu(X)$.  We say that $\Upsilon_X$ is
 \emph{compatible with $\jmath_X$}  if $$\d\Upsilon_X|_X=\left(\begin{array}{cc} I & \mathcal{A} \\ 0 & \jmath_X^{-1} \end{array}\right),$$
where $I:TX\rightarrow TX$ is the identity and $\mathcal{A}:\Lambda^2_+T^*X\rightarrow TX$ is arbitrary.
\end{dfn}

We now construct our tubular neighbourhoods using self-dual 2-forms as in the author's earlier papers \cite{LotayCS,Lotaydesing,LotayAC}, however our presentation is different and more in the style of \cite[$\S$3-4]{JoyceCS4}
 as it is more convenient.  

Observe that for $t>0$ we have a dilation map $\delta_t:C\rightarrow C$ given by $\delta_t(r,\sigma)=(tr,\sigma)$, which is a diffeomorphism, and its inverse is 
$\delta_{t^{-1}}$.   
We thus have an isomorphism $\delta_{t^{-1}}^*:\Lambda^2_+T^*_{(r,\sigma)}C\rightarrow \Lambda^2_+T^*_{(tr,\sigma)}C$.  Moreover, if 
$\alpha\in C^{\infty}(\Lambda^2_+T^*C)$ then
$|\alpha|_{t^2g_C}=t^{-2}|\alpha|_{g_C}$  and so 
$|t^3\alpha|_{t^2g_C}=t|\alpha|_{g_C}$.  We deduce that $\delta_{t^{-1}}^*$ scales the lengths of self-dual 2-forms by a factor of $t^{-2}$ and  the natural dilation action on $\Lambda^2_+T^*C$ is given by:
\begin{equation}\label{dilationeq}
\big(r,\sigma,\alpha 
\big)\mapsto \big(tr,\sigma,t^3\delta_{t^{-1}}^*\alpha \big).
\end{equation}

\begin{prop}\label{Cnbdprop} 
\begin{itemize}\item[]
\item[\emph{(a)}] There exist dilation-invariant open neighbourhoods 
$U_C\subseteq\Lambda^2_+T^*C$ and $T_C\subseteq \R^7$ of $C$, with $U_C$ given by 
$$U_C=\{\big(r,\sigma,\alpha(r,\sigma)\big)\,:\,|\alpha|<2\zeta r\}$$
for some $\zeta>0$, and a
 dilation-equivariant diffeomorphism $\Upsilon_C:U_C\rightarrow T_C$ such that $\Upsilon_C|_C=\id_C$ and is compatible with $\jmath_C$. 
\item[\emph{(b)}] Make $\epsilon$ smaller and $R$ larger if necessary so that 
$$|\Phi_N(r,\sigma)-r\sigma|<\zeta r\text{ for all }r<\epsilon\quad\!\!\text{and}\quad\!\! |\Phi_A(r,\sigma)-r\sigma|<\zeta r\text{ for all }r>R.$$ There exist self-dual 2-forms $\alpha_N$ on $(0,\epsilon)\times\Sigma$ 
and $\alpha_A$ on $(R,\infty)\times\Sigma$ 
such that 
$$
\Upsilon_C\big(r,\sigma,\alpha_N(r,\sigma)\big)=\Phi_N(r,\sigma)-r\sigma\text{ and } \Upsilon_C\big(r,\sigma,\alpha_A(r,\sigma)\big)=\Phi_A(r,\sigma)-r\sigma.
$$
Moreover, for all $j\in\N$, 
$$|\nabla_C^j\alpha_N|=O(r^{\mu-j})\;\,\text{as }r\rightarrow 0\quad\text{and}\quad |\nabla_C^j\alpha_A|=O(r^{\lambda-j})\;\,
\text{as }r\rightarrow\infty.$$
\end{itemize}
\end{prop}  

\begin{proof}  Detailed arguments are given in \cite{LotayCS,LotayAC} so we just sketch the idea here.

Applying the Tubular Neighbourhood Theorem to $\Sigma\subseteq\mathcal{S}^6$, we can easily construct diffeomorphic
 dilation-invariant open neighbourhoods of $C$ in $\nu(C)$ and $\R^7$.  Using $\jmath_C$ gives (a).   
 
We can certainly change $\epsilon$ and $R$ as claimed given the asymptotic behaviour of $\Phi_N$ and $\Phi_A$. Part (b) then follows from (a), the definition of $N$ and $A$ as CS and AC submanifolds and the fact that $\jmath_C$ is an isometric isomorphism.
\end{proof}

Proposition \ref{Cnbdprop} says we may effectively view the ends of $A$ and $N$ as graphs of the self-dual 2-forms $\alpha_A$ and $\alpha_N$ on the cone.
We can then extend this result  as in our earlier work to give neighbourhoods of $A$ and $N$, which are 
adapted so that we may realize graphs of small self-dual 2-forms on 
the ends as graphs of small self-dual 2-forms on the cone.  

\begin{prop}\label{NAnbdprop} Recall the notation of Proposition \ref{Cnbdprop}.
\begin{itemize}
\item[\emph{(a)}] There exist open neighbourhoods $U_N\subseteq \Lambda^2_+T^*\hat{N}$ and $T_N\subseteq M$ of $\hat{N}$ and a
 diffeomorphism $\Upsilon_N:U_N\rightarrow T_N$ such that $\Upsilon_N|_N=\id_N$ and is compatible with
 $\jmath_N$.  Further, 
$$\Psi_N^*(U_N)=\{\big(r,\sigma,\alpha(r,\sigma)\big)\,:\,r<\epsilon,\,|\alpha|<\zeta r\}$$ and
$$\Upsilon_N\big(\Psi_N(r,\sigma),\alpha\big(\Psi_N(r,\sigma)\big)\big)=
\chi\circ\Upsilon_C\big(r,\sigma,\alpha_N(r,\sigma)+\Psi_N^*\alpha(r,\sigma)\big).$$
\item[\emph{(b)}]
There exist open neighbourhoods $U_A\subseteq \Lambda^2_+T^*A$ and $T_A\subseteq\R^7$ of $A$ and a diffeomorphism 
$\Upsilon_A:U_A\rightarrow T_A$ such that $\Upsilon_A|_A=\id_A$ and is compatible with $\jmath_A$.  Further, 
$$\Phi_A^*(U_A)=\{\big(r,\sigma,\alpha(r,\sigma)\big)\,:\,r>R,\,|\alpha|<\zeta r\}$$ and
$$\Upsilon_A\big(\Phi_A(r,\sigma),\alpha\big(\Phi_A(r,\sigma)\big)\big)=
\Upsilon_C\big(r,\sigma,\alpha_A(r,\sigma)+\Phi_A^*\alpha(r,\sigma)\big).$$
\end{itemize}
\end{prop} 

\begin{remark} Observe the important difference between (a) and (b): in (a) we need to use the particular identification $\chi$ between an open ball
 in $\R^7$ and an open neighbourhood of $z$ in $M$.
 \end{remark}

Having identified self-dual 2-forms $\alpha$ with nearby submanifolds $X_{\alpha}$ to a coassociative 4-fold $X$ we 
may ask: what is the condition on $\alpha$ which makes $X_{\alpha}$ coassociative?  By Definition \ref{coassdfn} this is given by $\varphi|_{X_{\alpha}}=0$, which
 leads to a fully nonlinear equation on $\alpha$.  By the calculation in \cite[p.~731]{McLean} we see that the 
linearisation of this equation is $\d\alpha=0$ since $\d\varphi=0$.  (Here is where we use the 
condition that $(M,\varphi)$ is an almost $\GG_2$ manifold, since otherwise the linearisation would have further terms.)  We deduce the following well-known fact.

\begin{prop}\label{infdefprop} Let $X$ be a coassociative 4-fold in an almost $\GG_2$ manifold.  Infinitesimal 
 coassociative deformations of $X$ are given by closed self-dual 2-forms on $X$.   
\end{prop}

Closed self-dual forms are trivially also coclosed.  Hence,
 if $X$ is compact, Hodge theory implies that such forms uniquely represent cohomology classes in $H^2(X)$. 
 In the non-compact setting we do not have such a result, but for AC and CS 4-folds we can say which cohomology classes are uniquely
 represented by $L^2$ closed self-dual 2-forms.
This leads to our next definition.
 
 \begin{dfn}\label{H2+dfn}
For any Riemannian 4-manifold $X$, let 
\begin{align*}
\mathcal{H}^2(X)&=\{\alpha\in L^2(\Lambda^2T^*X):\d\alpha=\d^*\alpha=0\},\\
\mathcal{H}^2_{\pm}(X)&=\{\alpha\in L^2(\Lambda^2_{\pm}T^*X):\d\alpha=0\}.
\end{align*}
Notice that $\mathcal{H}^2(X)=\mathcal{H}^2_+(X)\oplus\mathcal{H}^2_-(X)$ and that by elliptic regularity 
$\mathcal{H}^2(X)$ consists of smooth forms.  

If $X$ is compact then $\dim\mathcal{H}^2(X)=b^2(X)$ and $\dim\mathcal{H}^2_{\pm}(X)=b^2_{\pm}(X)$.  If $X$ is an AC or (the nonsingular part of) a CS 4-fold and we let $$\mathcal{J}(X)=\Imm\big(H^2_{\text{cs}}(X)\rightarrow H^2(X)\big)$$ then, 
by \cite[Examples (0.15) \& (0.16)]{Lockhart},
$$\dim\mathcal{H}^2(X)=\dim\mathcal{J}(X)\quad\text{and}\quad\dim\mathcal{H}^2_{\pm}(X)=\dim\mathcal{J}_{\pm}(X),$$ where
 $\mathcal{J}_{\pm}(X)$ are the maximal positive and negative subspaces of $\mathcal{J}(X)$ with respect to the cup product.  (The subspaces 
$\mathcal{J}_{\pm}(X)$ are well-defined because the cohomology classes in $\mathcal{J}(X)$ are represented by compactly supported
 forms.)  We thus define $b^2_{\pm}(X)=\dim\mathcal{J}_{\pm}(X)$.
\end{dfn}

By \cite[$\S$4]{McLean}, the deformation theory of compact coassociative 4-folds $X$ is unobstructed, so infinitesimal deformations 
always extend to genuine deformations and thus we have the following. 

\begin{thm}\label{cptdefthm} Let $X$ be a compact coassociative 4-fold in an almost $\GG_2$ manifold.
The moduli space of compact coassociative deformations of $X$ is a smooth manifold near $X$ of dimension $b^2_+(X)$.
\end{thm}
  
\noindent The author extended this result in \cite{LotayCS} and \cite{LotayAC} to the CS and AC settings, where various similarities 
 and differences occur which shall be discussed later.  These results will be crucial in understanding obstructions to the gluing problem.
 
\section{Desingularization: geometry}\label{geometry}

In this section we tackle the more ``geometric'' aspects of our desingularization problem.  The key part is to 
construct 
an appropriate connect sum $\tilde{N}(t)$ of $\hat{N}$ and $tA$ such that $\tilde{N}(t)$ is a smooth compact 
4-fold with $\tilde{N}(t)\rightarrow N$ as $t\rightarrow 0$.  The crucial point will be to ensure that $\tilde{N}(t)$ is sufficiently ``close'' to being 
coassociative; i.e.~that $|\varphi|_{\tilde{N}(t)}|$ is ``small enough'' that one may hope to perturb $\tilde{N}(t)$ to a nearby coassociative 4-fold $N(t)$.
 
Unlike in \cite{Lotaydesing}, where one was able to construct $\tilde{N}(t)$ 
using a rather naive connect sum, 
  here we have to use a more refined technique which requires a detailed understanding 
of the geometric \emph{obstructions} to the coassociative gluing procedure.  We discover that the obstructions which emerge are both topological and analytic 
in nature, and we can give natural interpretations for the obstructions.  

\subsection{Obstructions}\label{obstruct}

Studying the argument in 
\cite{Lotaydesing}, one sees that for our 
problem we simply cannot use the same
 method of constructing $\tilde{N}(t)$ since the analysis will fail.  This is not a flaw with the analytic method, but rather it is
 a \emph{geometric} phenomenon.  Specifically, in \cite{Lotaydesing} geometric assumptions were made precisely to ensure that the desingularization was \emph{unobstructed}. 
In general, there are geometric \emph{obstructions} to resolving the coassociative conical singularity, which we now identify.

We begin by examining the cone $C$.  Consider a self-dual 2-form $\alpha$ on $C$ which is homogeneous of rate $\upsilon$ say.  We may write
\begin{equation*}
\alpha=r^{\upsilon+2}(\alpha_{\Sigma}+r^{-1}\d r\w *_{\Sigma}\alpha_{\Sigma})
\end{equation*}
for a 2-form $\alpha_{\Sigma}$ on the link $\Sigma$ of $C$, noting that $|\alpha_{\Sigma}|_{g_C}=O(r^{-2})$ and $|*_{\Sigma}\alpha_{\Sigma}|_{g_C}=O(r^{-1})$.  (We use the notation $*_\Sigma$ to clarify that we are using the Hodge star on $\Sigma$.)  The condition that $\alpha$ is closed is equivalent to 
\begin{equation}\label{Deq}
\d\!*_{\Sigma}\!\alpha_{\Sigma}=(\upsilon+2)\alpha_{\Sigma}\quad\text{and}\quad\d\alpha_{\Sigma}=0.
\end{equation}
Such closed forms $\alpha$ define infinitesimal coassociative deformations of $C$ by Proposition \ref{infdefprop}.  These forms 
will also naturally relate to deformations of $A$ and $\hat{N}$.  To understand this relationship we first make a convenient definition.

\begin{dfn}\label{Ddfn}
For $\upsilon\in\R$ let $D(\upsilon)\subseteq C^{\infty}(\Lambda^2T^*\Sigma)$ be the space of solutions to \eq{Deq}, so  
$D(\upsilon)$ corresponds to the homogeneous closed self-dual 
2-forms on $C$ of rate $\upsilon$.  We also let 
$\mathcal{D}=\big\{\upsilon\in\R\,:\,D(\upsilon)\neq\{0\}\big\}$ and let $d_{\mathcal{D}}(\upsilon)=\dim D(\upsilon)$. 
\end{dfn}

\begin{remarks}
\begin{itemize}\item[]
\item[(a)]The set $\mathcal{D}$ is countable and discrete, and $d_{\mathcal{D}}(\upsilon)$ is always finite. 
\item[(b)]
For $\upsilon=-2$, \eq{Deq} is equivalent to the statement that $\alpha_{\Sigma}$ is closed and coclosed.  Thus 
$d_{\mathcal{D}}(-2)=b^1(\Sigma)$ and $-2\in\mathcal{D}$ if and only if $b^1(\Sigma)\neq 0$.
\item[(c)] For $\upsilon=1$, \eq{Deq} gives closed self-dual 2-forms on $C$ which are invariant under dilations 
and thus correspond to infinitesimal deformations of $C$ as a coassociative cone, which include $\GG_2$ transformations of $C$.  
Similarly, for $\upsilon=0$, \eq{Deq} defines infinitesimal coassociative deformations 
of $C$ which are homogeneous of order $0$, which include translations of $C$.
\end{itemize}
\end{remarks}

Since we may view $A$ as a manifold with boundary $\Sigma$, we have the following exact sequence:
\begin{equation}\label{Aseq}
\cdots\longrightarrow H^m_{\text{cs}}(A)
\,{\buildrel\iota^A_m\over\longrightarrow}\,
H^m(A)\,{\buildrel j^A_m\over\longrightarrow}\,
H^m(\Sigma)\,{\buildrel
\partial^A_m\over\longrightarrow}\,
H_{\text{cs}}^{m+1}(A)\longrightarrow\cdots.
\end{equation}
The connection between deformations of $C$ and $A$ can now be succinctly expressed through one of the main results in \cite{LotayAC}.

\begin{thm}\label{ACdefthm}  Suppose that the rate $\lambda<0$ and let $\lambda_+\in(-2,0)\setminus\mathcal{D}$ be such that
 $\lambda_+\geq\lambda$.  The moduli space of deformations of $A$ as an AC
 coassociative 4-fold with rate $\lambda_+$ and cone $C$ is a smooth manifold near $A$ of dimension
$$b^2_+(A)+\dim\Imm j_2^A
+\!\!\!\!\!\!
\sum_{\upsilon\in(-2,\lambda_+)}\!\!\!\!\!\!d_{\mathcal{D}}(\upsilon),$$
which is the dimension of 
$$
\{\alpha\in C^{\infty}(\Lambda^2_+T^*A)\,:\,\d\alpha=0,\,|\nabla_C^j\Phi_A^*\alpha|=O(r^{\lambda_+-j})\text{ as }r\rightarrow\infty\text{ for all }j\in\N\}.$$
\end{thm}

The appearance of the term $b^2_+(A)$ is clear by Definition \ref{H2+dfn}.  
A key part of the proof and dimension count relies on showing that various closed self-dual 2-forms on $C$ lift to $A$, applying the theory 
in \cite{LockhartMcOwen}.  More precisely, given a homogeneous closed self-dual 2-form $\alpha_C$ on $C$ defined by a solution to \eq{Deq} for 
$\upsilon\in(-2,0)$, one needs to show that there exists a closed self-dual 2-form $\alpha$ on $A$ such that $\alpha$ is asymptotic to $\alpha_C$ in the sense
 that, for some $\delta>0$, 
$$|\nabla^j_C(\Phi_A^*\alpha-\alpha_C)|=O(r^{\upsilon-\delta-j})\text{ as }r\rightarrow\infty\text{ for all }j\in N.$$
The forms on $C$ corresponding to forms in $\mathcal{H}^2_+(A)$ are actually zero, 
but for the other terms in the dimension count one has non-trivial forms on $C$ lifting to $A$ in the sense just described.
   
Specifically, the harmonic representatives of 
the classes in $\Imm j^A_2$ define the homogeneous closed self-dual 2-forms on $C$ with order $O(r^{-2})$
 which lift to define closed self-dual 2-forms on $A$.  Notice that such forms on $A$, given their decay rate on the ends, cannot lie 
in $L^2$ and so do not contribute to $b^2_+(A)$.

Moreover, the sum over $d_{\mathcal{D}}(\upsilon)$ counts the homogeneous 
 closed self-dual 2-forms on $C$ which have rate between $-2$ and $\lambda_+$, and the proof of Theorem \ref{ACdefthm} shows that
these forms on $C$  all lift to closed self-dual 2-forms on $A$.  

The final key part of the proof of Theorem \ref{ACdefthm} is to show that, given a closed self-dual 2-form $\alpha^0$ on $A$ with appropriate decay on 
the ends, one can solve for a transverse self-dual 2-form $\alpha^{\prime}$ on $A$ so that $\varphi_0$ vanishes on the graph of $\alpha^0+\alpha^{\prime}$ via
 the Implicit Function Theorem.  Thus we can extend the infinitesimal AC coassociative deformation of $A$ given by $\alpha^0$ to a genuine deformation.

With these preliminaries we are now able to analyse $A$ further.
 
\begin{prop}\label{alphaAprop} Suppose that $\lambda<0$ and let
$$\mathcal{K}_C(\lambda)=
\Span\{r^{\upsilon+2}(\alpha_{\Sigma}+r^{-1}\d r\w*_{\Sigma}\alpha_{\Sigma})\,:\,\alpha_{\Sigma}\in D(\upsilon)\,,\,\upsilon\in[-2,\lambda]\} 
$$ 
if $\lambda\geq -2$ and set $\mathcal{K}_C(\lambda)=\{0\}$ if $\lambda<-2$.
 
The form $\alpha_A$ over $(R,\infty)\times\Sigma$ given in Proposition \ref{Cnbdprop}(b) can be decomposed into self-dual 2-forms as 
$\alpha_A=\alpha_A^0+\alpha_A^{\prime}$, where $\alpha_A^0\in \mathcal{K}_C(\lambda)$
 and $\alpha_A^{\prime}$ is transverse to $\mathcal{K}_C(\lambda)$ and satisfies, for some $\lambda_-<-2$,
$$|\nabla^j_C\alpha_A^{\prime}|=O(r^{\max\{2\lambda-1,\lambda_-\}-j})\text{ as }r\rightarrow\infty\text{ for all }j\in\N.$$
\end{prop}

\begin{proof} If $\lambda<-2$ then we may choose $\alpha_A^0=0$ and $\lambda_-=\lambda$, so suppose 
from now on that $\lambda\in[-2,0)$.

  Since $A$ is coassociative, $\varphi_0$ vanishes on the graph of $\alpha_A$.  By \cite[Proposition 4.2]{McLean} and the compatibility
 conditions we have imposed on $\Upsilon_C$, we see as in the proof of \cite[Proposition 6.9]{LotayCS} that 
\begin{equation}\label{phiexpandeq}
\varphi_0\big(\Upsilon_C(r,\sigma,\alpha_A(r,\sigma))\big)
=\d\alpha_A(r,\sigma)+P_C\big(r,\sigma,\alpha_A(r,\sigma),\nabla_C\alpha_A(r,\sigma)\big),
\end{equation}
where $$\big|\nabla^j_CP_C\big(r,\sigma,\alpha_A(r,\sigma),\nabla_C\alpha_A(r,\sigma)\big)\big|=O(r^{2\lambda-2-j}).$$  Here we have used the fact that 
$r^{-1}|\alpha_A|$ and $|\nabla_C\alpha_A|$ tend to zero as $r\rightarrow\infty$.   

Let $\alpha_A^0$ be the projection of $\alpha_A$ onto $\mathcal{K}_C(\lambda)$ and let $\alpha_A^{\prime}=\alpha_A-\alpha_A^0$.  Closed self-dual 2-forms 
on $C$ can be written as linear combinations of forms of the following type:
$$r^{\upsilon+2}\big(\beta_{\Sigma}(r,\sigma)+r^{-1}\d r\wedge *_{\Sigma}\beta_{\Sigma}(r,\sigma)\big)$$
where $\beta_{\Sigma}$ is a polynomial in $\log r$ which takes values in $C^{\infty}(\Lambda^2T^*\Sigma)$ and $\upsilon\in\mathcal{D}$.  This fact, which is true
 more generally for solutions of suitable elliptic equations, is a key part of the work in \cite{LockhartMcOwen}.  (Although cylinders are discussed in \cite{LockhartMcOwen} rather 
than cones, one can transform between these situations in a natural way.)  By \cite[Proposition 5.8]{LotayAC}, in fact 
$\beta_{\Sigma}(r,\sigma)=\alpha_{\Sigma}(\sigma)$ is independent of $r$ and $\alpha_{\Sigma}$ satisfies \eq{Deq}.  Hence, the condition that $\alpha_A^{\prime}$ 
is transverse to $\mathcal{K}_C(\lambda)$ means that $\alpha_A^{\prime}$ is transverse to the closed self-dual forms on $C$ for rates $\upsilon\in[-2,0)$.

 Now $\d\alpha_A^0=0$ so 
$$\d\alpha_A^{\prime}(r,\sigma)=-P_C\big(r,\sigma,\alpha_A(r,\sigma),\nabla_C\alpha_A(r,\sigma)\big)$$ and thus
 $|\nabla^j_C\d\alpha_A^{\prime}|=O(r^{2\lambda-2-j})$ for all $j\in\N$.  Since $\alpha_A^{\prime}$ is transverse to
the closed self-dual 2-forms on $C$ for rates in $[-2,0)$ and $2\lambda-2\neq -1$, 
we may integrate and choose $\alpha_A^{\prime}$ so that $|\alpha_A^{\prime}|=O(r^{2\lambda-1})$.  
\end{proof}

\begin{remark}
We see that if $\lambda<-2$ then Proposition \ref{alphaAprop} is irrelevant.  This proposition marks the 
significant departure from the work in \cite{Lotaydesing} where we restricted ourselves to the case $\lambda<-2$.
\end{remark}

We now make some observations to understand the obstructions to the gluing procedure.  
If we desingularize $N$ using $tA$ we will obtain a smooth 4-dimensional 
submanifold $\tilde{N}(t)$ of $M$ which is diffeomorphic to the disjoint union of the compact sets $tK_A$ and $K_N$ and the portion of the cone $(tR,\epsilon)\times\Sigma$.  

Suppose we construct $\tilde{N}(t)$ so that $|\varphi|_{\tilde{N}(t)}|=O(t^\eta)$ as $t\rightarrow 0$.  Clearly we need $\eta>0$ so that $\tilde{N}(t)$ converges 
to the coassociative 4-fold $N$ as $t\rightarrow 0$, but we also need $\eta$ sufficiently large to make $|\varphi|_{\tilde{N}(t)}|$ small enough as $t\rightarrow
0$ so that the effect of $\tilde{N}(t)$ becoming singular is dominated by the rate at which $\tilde{N}(t)$ is becoming coassociative.

We can view the subset of $\tilde{N}(t)$ which is diffeomorphic to $(tR,\epsilon)\times\Sigma$ as the graph of 
a self-dual 2-form $\alpha$.
Since we are using $tA$ to construct $\tilde{N}(t)$ we need to understand the behaviour of $\alpha$ as $t\rightarrow 0$.  Recall that we 
identified the ends of $A$ with the graph of a self-dual 2-form $\alpha_A$ over $(R,\infty)\times\Sigma$ in Proposition 
\ref{Cnbdprop}.  
Observe that, from this identification,
we may write $tA$ as the graph of $\alpha_{tA}=t^3\delta_{t^{-1}}^*\alpha_A$ over $(tR,\infty)\times\Sigma$, recalling the dilation action \eq{dilationeq} on 
$\Lambda^2_+T^*C$.   If $\alpha_A$ were homogeneous of rate $\lambda$, then
\begin{equation}\label{alphatA}
\alpha_{tA}=t^3\delta_{t^{-1}}^*\alpha_A=
t^3\delta_{t^{-1}}^*\big(r^{\lambda+2}(\alpha_{\Sigma}+r^{-1}\d r\wedge*_{\Sigma}\alpha_{\Sigma})\big)
=t^3(t^{-1})^{\lambda+2}\alpha_A=t^{1-\lambda}\alpha_A.
\end{equation}
More generally, the condition that $\alpha_A$ is of order $O(r^{\lambda})$ implies that  $\alpha_{tA}$ is of order $O(t^{1-\lambda}r^{\lambda})$.  

Naively, we would 
construct $\alpha$ by interpolating between $\alpha_{tA}$ and $\alpha_N$ given 
in Proposition \ref{Cnbdprop} over $I\times\Sigma$ for some suitable chosen interval $I\subseteq (tR,\epsilon)$ so that if $r\in I$ then $r=O(t^{\nu})$ for some 
$\nu\in[0,1]$.   From this choice of $I$,
the contribution to the behaviour of $\alpha$  from $\alpha_{tA}$ is of order $O(t^{1-\lambda}r^{\lambda})$ and from $\alpha_N$ is of order $O(r^{\mu})$.

Using a similar equation to \eq{phiexpandeq}, naively the behaviour of $|\varphi|_{\tilde{N}(t)}|$
 is dominated by $|\d\alpha|$, which we can estimate on $I\times\Sigma$ 
using terms of order $O(t^{1-\lambda}r^{\lambda-1})=O(t^{(1-\lambda)(1-\nu)})$ from $\alpha_A$ and 
terms of order $O(r^{\mu-1})=O(t^{\nu(\mu-1)})$ from $\alpha_N$.  Since $\lambda<1$ and $\mu>1$, we would naturally require that $0<\nu<1$ to ensure that  
$|\varphi|_{\tilde{N}(t)}|=O(t^{\eta})$ for $\eta>0$.  

However, we now observe that if we take $r\in I$ 
and consider the natural inclusion $\iota_r:\Sigma\rightarrow\tilde{N}(t)$ of $\{r\}\times\Sigma$ in 
$\tilde{N}(t)$, we see that to have any hope of ensuring that we can desingularizing $N$ with $tA$ we would need that 
$$[\varphi|_{\tilde{N}(t)}]\cdot[\iota_r(\Sigma)]=\int_{\iota_r(\Sigma)}\varphi\rightarrow 0$$ as $t\rightarrow 0$.  We can calculate that, if
$|\varphi|_{\tilde{N}(t)}|=O(t^{\eta})$ then
$$\int_{\iota_r(\Sigma)}\varphi 
=O(t^{\eta-3}r^3)=O(t^{\eta-3(1-\nu)}),$$
where the factor $t^{-3}$ appears 
because the metric on the interpolation region in $\tilde{N}(t)$ blows up as $t\rightarrow 0$ 
since $\tilde{N}(t)$ becomes singular.  
This suggests that we need $\eta>3(1-\nu)$ to ensure that $[\varphi|_{\tilde{N}(t)}]\cdot[\iota_r(\Sigma)]\rightarrow 0$ as $t\rightarrow 0$.  
  
  If we just use the estimates we had before, we see that we would require that $(1-\lambda)(1-\nu)>3(1-\nu)$ and $\nu(\mu-1)>3(1-\nu)$ to achieve 
$\eta>3(1-\nu)$.  The condition involving $\mu$
is equivalent to $\nu>\frac{3}{\mu+2}$ which, since we think of $\mu>1$ as being close to $1$, means that we just need to choose $\nu$ sufficiently close to $1$. 
Notice that taking $\nu$ closer to $1$ will ensure that the connect sum occurs over a smaller region and 
$\tilde{N}(t)$ is closer to the initial CS coassociative $N$ as a submanifold, and hence is a natural constraint.  However, the equation
 $(1-\lambda)(1-\nu)>3(1-\nu)$ can only hold for $\lambda<-2$ since $\nu\in(0,1)$.

So it would appear that we need $\lambda<-2$ for $|\varphi|_{\tilde{N}(t)}|$ to be small enough for the desingularization to succeed.  
 (This is a way to interpret how this condition arises in \cite{Lotaydesing}.)  Moreover, this suggests that for rates $\lambda\geq -2$ we should
 see obstructions to our gluing procedure, whereas for $\lambda<-2$ we should not.

However, if we can arrange $\alpha$ to be \emph{closed} then, again using an equation like \eq{phiexpandeq}, we have that 
$|\varphi|_{\tilde{N}(t)}|$ is now dominated by the nonlinear terms in $\alpha$ and its derivatives which, roughly speaking, 
are then bounded by $|r^{-1}\alpha|^2$ and $|\nabla_C\alpha|^2$.  One sees that these terms are of order 
$O(t^{2(1-\lambda)(1-\nu)})$ and $O(t^{2\nu(\mu-1)})$, so
we have thus improved our estimate drastically as we now only require $2-2\lambda>3$ for our analysis to go through, 
which is equivalent to $\lambda<-\frac{1}{2}$. (This 
in part can be seen from Proposition \ref{alphaAprop}, where now $2\lambda-1<-2$, so $\max\{2\lambda-1,\lambda_-\}<-2$.)
  
We deduce that for rates $\lambda\in[-2,-\frac{1}{2})$, the obstructions to the desingularization 
should arise purely from the ability to glue $tA$ and $N$ using a closed 
self-dual 2-form, which is a natural constraint in the context of coassociative geometry.  

These considerations allow us  to state the key condition that we require to overcome the obstructions.
Notice that since $\alpha_A^0$ is homogeneous it is defined on the entire cone $C$.  

\begin{dfn}\label{matchingdfn}
Let $\mathcal{D}\cap[-2,\lambda]=\{\lambda_1,\ldots,\lambda_d\}$ with $\lambda_1<\ldots<\lambda_d$.   Write $\alpha_A^0$ given by Proposition \ref{alphaAprop} as
$\alpha_A^0=\sum_{i=1}^d\alpha_A^i$, so there are $\alpha_{\Sigma}^i\in D(\lambda_i)$ such that
$$\alpha_A^i(r,\sigma)=r^{\lambda_i+2}\big(\alpha_{\Sigma}^i(\sigma)+r^{-1}\d r\w*_{\Sigma}\alpha_{\Sigma}^i(\sigma)\big).$$  
We say that $A$ and $N$ satisfy the \emph{matching condition} if
 there exists $\delta_0>0$ and for $i=1,\ldots,d$ there exists a closed self-dual form 
$\alpha_N^i$ on $\hat{N}$ such that 
$$\big|\nabla_C^j\big(\Psi_N^*\alpha_N^i(r,\sigma)-\alpha_A^i(r,\sigma)\big)\big|=O(r^{\lambda_i+\delta_0-j})\text{ as }r\rightarrow 0\text{ for all }j\in\N.$$
Effectively, this says that each infinitesimal coassociative deformation $\alpha^i_A$ of the cone $C$, which we know lifts to an infinitesimal coassociative 
deformation of $A$ by Theorem \ref{ACdefthm}, also lifts to an 
infinitesimal coassociative deformation $\alpha^i_N$ of $\hat{N}$ so that  to ``leading order'' $\alpha_N^i$ tends to $\alpha_A^i$ as $r\rightarrow 0$.
\end{dfn}

\begin{remarks}
\begin{itemize}
\item[]
\item[(a)]
As we shall see, the matching condition precisely allows us to define our desingularization 
so that over $(tR,\epsilon)\times\Sigma$ it is the 
graph of a self-dual 2-form whose leading order term is closed.
\item[(b)] Notice from \eq{alphatA} that under the action of dilation in \eq{dilationeq} we see that 
$\alpha_A^i\mapsto t^3\delta_{t^{-1}}^*\alpha_A^i=t^{1-\lambda_i}\alpha_A^i$.
\end{itemize}
\end{remarks}

We now give further geometric meaning for (part of) our matching condition.  
 Notice that if $\upsilon\neq -2$ and $\alpha_{\Sigma}\in D(\upsilon)$ then
 $\alpha_{\Sigma}$ is exact, whereas if $\alpha_{\Sigma}\in D(-2)$ then $\alpha_{\Sigma}$ uniquely determines a cohomology class in 
 $H^2(\Sigma)$.  Hence, there exists unique $\alpha_{\Sigma}^0\in D(-2)$ such that 
$[\alpha_A^0]=[\alpha_\Sigma^0]\in H^2(\Sigma)\cong H^2(C)$.  

In the notation of Definition \ref{matchingdfn}, we have that $\alpha_{\Sigma}^0=0$ if $\lambda_1>-2$ and $\alpha_{\Sigma}^0=\alpha_{\Sigma}^1$ if 
$\lambda_1=-2$.  The class $[\alpha_{\Sigma}^0]$ is not mysterious but has a natural 
geometric interpretation.  Recall the map $\jmath_A$ given by Proposition \ref{jmathprop} and \eq{Aseq}.

\begin{prop}\label{infdilprop} 
Let $v$ be the
dilation vector field on $\R^7$ and let $u$ be the projection of $v$ onto the normal bundle of $A$.  We have that $\d\jmath_A(u)=0$,
$$[\jmath_A(u)]=3[\varphi_0]\in H^3(\R^7;A)\cong H^2(A)\quad\text{and}$$ 
$$j^A_2[\jmath_A(u)]=3[\alpha_\Sigma^0]\in H^2(\Sigma).$$
\end{prop}

\begin{remark}
Proposition \ref{infdilprop} says that the cohomology class of the infinitesimal dilation deformation of $A$
is a multiple of the class of $\varphi_0$ in $H^3(\R^7;A)\cong H^2(A)$, 
which itself projects to the class of $\alpha_\Sigma^0$ in $H^2(\Sigma)$.
\end{remark}

\begin{proof}
First $\jmath_A(u)=u\lrcorner\varphi_0|_A=v\lrcorner\varphi_0|_A$ since 
$\varphi_0|_A=0$.  Now $\varphi_0$ is homogeneous of degree 3 so $\d(v\lrcorner\varphi_0)=\mathcal{L}_v\varphi_0=3\varphi_0$.  
Hence $\d\jmath_A(u)=3\varphi_0|_A=0$.  From this formula one also deduces that $[\jmath_A(u)]=3[\varphi_0]\in H^3(\R^7;A)$.  

The dilation deformation, $A\mapsto tA$ for $t>0$, is defined by a self-dual 2-form on $A$.  Over $C$ the dilation is given by 
$\alpha_A\mapsto \alpha_{tA}=t^3\delta_{t^{-1}}^*\alpha_A$ from \eq{dilationeq}.  Thus, the dilation deformation is defined by 
$\frac{\d}{\d t}|_{t=1}\alpha_{tA}=3\alpha_A$ over $C$.  Hence, the corresponding infinitesimal deformation is the 
closed part of $3\alpha_A$ which, to leading order (that is, for order at least $O(r^{-2})$), is given by $3\alpha_A^0$.  Now $\jmath_A(u)$ also defines the infinitesimal dilation deformation, so 
$\Phi_A^*\jmath_A(u)=3\alpha_A^0$ plus terms with order strictly less than $-2$.  Hence $j^A_2[\jmath_A(u)]=3[\alpha_A^0]=3[\alpha_{\Sigma}^0]$ as claimed.

\end{proof}

As for $A$ in \eq{Aseq}, we have an exact sequence for $\hat{N}$:
\begin{equation}\label{Nseq}
\cdots\longrightarrow H^m_{\text{cs}}(\hat{N})
\,{\buildrel\iota^N_m\over\longrightarrow}\,
H^m(\hat{N})\,{\buildrel j^N_m\over\longrightarrow}\,
H^m(\Sigma)\,{\buildrel
\partial^N_m\over\longrightarrow}\,
H_{\text{cs}}^{m+1}(\hat{N})\longrightarrow\cdots.
\end{equation}
We can now interpret part of the matching condition in topological terms.

\begin{prop}\label{alphaSigmaprop}
Suppose that $\alpha_{\Sigma}$ is a harmonic representative of a cohomology class in $\Imm j^N_2$. There exists a closed self-dual 2-form $\alpha$ on $\hat{N}$ and $\delta_0>0$ 
such that, for all $j\in\N$,  
$$\big|\nabla^j_C\big(\Psi_N^*\alpha(r,\sigma)-(\alpha_{\Sigma}(\sigma)+r^{-1}\d r\w *_{\Sigma}\alpha_{\Sigma}(\sigma))\big)\big|=O(r^{-2+\delta_0-j})
\text{ as }r\rightarrow
0.$$
\end{prop}

\begin{proof}
Since $\alpha_{\Sigma}$ is a closed 2-form on $C$ such that $[\alpha_{\Sigma}]\in\Imm j^N_2$, we can pull it back to
the end of $\hat{N}$ and extend it smoothly to define a closed 2-form $\beta$ 
(c.f. \cite[Proposition 5.8 \& Corollary 5.9]{Marshall}) so that $|\nabla^j_C\Psi_N^*\beta|=O(r^{-2-j})$ as $r\rightarrow 0$ 
for all $j\in\N$. 
 Moreover $\alpha_{\Sigma}$ is coclosed, so since the metric on $\hat{N}$ converges to $g_C$ with rate $O(r^{\mu-1})$ and a dilation-invariant 3-form on $C$ 
has order $O(r^{-3})$, we have that 
$|\nabla^j_C\Psi_N^*\d\!*\!\beta|=O(r^{\mu-4-j})$ as $r\rightarrow 0$ for all $j\in\N$.  
Thus, if $\gamma=\beta+*\beta$, 
$$\big|\nabla^j_C\big(\Psi_N^*\gamma(r,\sigma)-(\alpha_{\Sigma}(\sigma)+r^{-1}\d r\w *_{\Sigma}\alpha_{\Sigma}(\sigma))\big)\big|=O(r^{\mu-3-j})
\text{ as }r\rightarrow 0.$$
Since $\mu-3>-2$ we would be done except that $\gamma$ is not necessarily closed.   

Now, $\d\gamma$ lies in the space of exact forms which decay at rate $O(r^{\mu-4})$, so lies in the image of $\d$ acting on 2-forms which 
decay with rate $O(r^{\mu-3})$.  As we shall see in $\S$\ref{stabsec}, given $k\geq 4$ and $\delta_0>0$ such that $(-2,-2+\delta_0]\cap\mathcal{D}=\emptyset$ we have that
$$\d\big(L^2_{k,-2+\delta_0}(\Lambda^2T^*\hat{N})\big)=\d\big(L^2_{k,-2+\delta_0}(\Lambda^2_+T^*\hat{N})\big)$$ 
 (see \cite[$\S$3.4]{Lotaydesing} for example for the definition of the 
weighted Sobolev spaces, which control the decay rate of forms near the singularity).  Choosing $-2+\delta_0<\mu-3$, there exists $\gamma^{\prime}\in L^2_{k,-2+\delta_0}(\Lambda^2_+T^*\hat{N})$ such that $\d\gamma^{\prime}=\d\gamma$.  Taking $\alpha=\gamma-\gamma^{\prime}$ 
and elliptic regularity gives the result.
\end{proof}

We deduce from Proposition \ref{alphaSigmaprop} that part of the matching condition is purely topological; that is, we can replace the condition that a closed 
self-dual 2-form exists on $\hat{N}$ asymptotic to the rate $-2$ part of $\alpha_A^0$ with the assumption that $[\alpha_\Sigma^0]$
 lies in $\Imm j_2^N$.  This motivates the following definition for convenience.
 
 \begin{dfn}\label{topmatchingdfn}  Recall \eq{Aseq} and \eq{Nseq}.  
 We say that $A$ and $N$ satisfy the \emph{topological matching condition} if 
$$j^A_2[\varphi_0]\in j^N_2\big(H^2(\hat{N})\big)\subseteq H^2(\Sigma),$$
 where $\varphi_0$ defines a cohomology class $[\varphi_0]\in H^3(\R^7;A)\cong H^2(A)$.
 \end{dfn}

We have identified part of the matching condition as a topological constraint, but the remainder is analytic and still needs to interpreted 
geometrically.  As we remarked, homogeneous closed self-dual 2-forms on $C$ with rates in $(-2,0)$ always extend to $A$.  
However, this is not the case for $\hat{N}$, and such forms which do not extend correspond to \emph{obstructions} to the deformation 
theory of $\hat{N}$ (c.f. \cite{LotayCS}).  Therefore if the deformation theory of $\hat{N}$ is \emph{unobstructed}, the analytic part of the 
matching condition will hold.  We shall discuss these ideas in detail later.

\medskip
The work in this subsection leads us to impose the following conditions on $A$ and $N$ from now on. 

\begin{conds*} Assume that
\begin{itemize}
\item the rate $\lambda$ of convergence of $A$ to $C$ satisfies $\lambda<-\frac{1}{2}$ and
\item $A$ and $N$ satisfy the matching condition in Definition \ref{matchingdfn}.
\end{itemize}
In particular, the topological matching condition in Definition \ref{topmatchingdfn} is satisfied.
\end{conds*}

From our discussion it is clear that, unless we make further assumptions or develop an even more sophisticated construction, 
the conditions we have imposed will be essential for our analysis to go through.  

 We shall see later that the condition $\lambda<-\frac{1}{2}$ allows for far more examples of conical singularities than the
 situation in \cite{Lotaydesing} where $\lambda<-2$.  One could conceivably impose further conditions on $\alpha$ for $\lambda\geq -\frac{1}{2}$
 to ensure that $|\varphi|_{\tilde{N}(t)}|$ has the required decay property for our analysis to work, but these appear to be less geometrically natural 
  so we choose not to pursue this.

\subsection{Construction}

We now define our (approximately coassociative) desingularizations $\tilde{N}(t)$ of $N$ using $tA$.  Recall the notation of the matching condition 
in Definition \ref{matchingdfn} and suppose without loss of generality that $\delta_0$ is small enough that 
\begin{equation}\label{delta0eq}
\mu>1+2\delta_0.
\end{equation}

\begin{dfn}\label{tildeNdfn} Let $f_{\text{inc}}:\R\rightarrow[0,1]$ be a smooth increasing function such that
$$f_{\text{inc}}(x)=\left\{\begin{array}{lll} 0 & & \text{for $x\leq 0$,}\\
1&&\text{for $x\geq1$}\end{array}\right.$$ and
$f_{\text{inc}}(x)\in(0,1)$ for $x\in(0,1)$.  Let $\nu>0$  
be such that
\begin{equation}\label{nueq}
\frac{1-\lambda}{1+\delta_0-\lambda}<\nu<1.
\end{equation}
Recall that we restrict the scale $t$ to lie in $(0,\tau)$.
Choose $\tau$ sufficiently small so that
$$0<\tau R<\frac{1}{2}\tau^{\nu}<\tau^{\nu}<\epsilon.$$
Let $\alpha_N^0(t)=\sum_{i=1}^dt^{1-\lambda_i}\alpha_N^i$, recall the dilation action on $\Lambda^2_+T^*C$ in \eq{dilationeq}, and
  define $\alpha_C(t)$ on $(tR,\epsilon)\times\Sigma$ by 
\begin{align}
\alpha_C(t)(r,\sigma)&=t^3\big(1-f_{\text{inc}}(2t^{-\nu}r-1)\big)\delta_{t^{-1}}^*\alpha_A(r,\sigma)\nonumber\\&\quad+f_{\text{inc}}(2t^{-\nu}r-1)\big(\Psi_N^*\alpha_N^0(t)(r,\sigma)+\alpha_N(r,\sigma)\big)\nonumber\\
&=t^3\delta_{t^{-1}}^*\alpha_A^0(r,\sigma)+t^3\big(1-f_{\text{inc}}(2t^{-\nu}r-1)\big)\delta_{t^{-1}}^*\alpha_A^{\prime}(r,\sigma)\nonumber\\
&\quad+f_{\text{inc}}(2t^{-\nu}r-1)\sum_{i=1}^d t^{1-\lambda_i}(\Psi_N^*\alpha_N^i(r,\sigma)-\alpha_A^i(r,\sigma))\nonumber\\&\quad+
f_{\text{inc}}(2t^{-\nu}r-1)\alpha_N(r,\sigma).\label{alphaCeq}
\end{align}
Observe that  
$$\alpha_C(t)(r,\sigma)=\left\{\begin{array}{ll}\alpha_{tA}(r,\sigma)=t^3\delta_{t^{-1}}^*\alpha_A(r,\sigma) & r\in(tR,\frac{1}{2}t^{\nu}),\\
\Psi_N^*\alpha_N^0(t)(r,\sigma)+\alpha_N(r,\sigma) & r\in(t^{\nu},\epsilon).\end{array}\right.$$
Therefore, if we let 
$$\tilde{N}(t)=\chi(tK_A)\cup\Upsilon_C(\Gamma_{\alpha_C(t)})\cup \Upsilon_N(\Gamma_{\alpha_N^0(t)|_{K_N}})$$
then, by Proposition \ref{NAnbdprop}, $\tilde{N}(t)$ is a smooth compact 4-fold so that $\tilde{N}(t)\rightarrow N$ as $t\rightarrow 0$
 in the sense of currents in Geometric Measure Theory.
\end{dfn}

\begin{remarks}
\begin{itemize}\item[]
\item[(a)] The precise choice of $\nu$ in \eq{nueq} will remain mysterious until late in the argument but, roughly, we need to choose $\nu$ close to $1$ so that 
the interpolation region $r\in[\frac{1}{2}t^{\nu},t^{\nu}]$, where $\varphi$ potentially has the worst behaviour, is small as $t\rightarrow 0$.
\item[(b)] The factors of $t^{1-\lambda_i}$ which appear in \eq{alphaCeq} are due to the fact that 
$\alpha_A^i$ maps to $t^{1-\lambda_i}\alpha_A^i$ under dilations, as remarked after Definition \ref{matchingdfn}.
\end{itemize}
\end{remarks}

Our aim is to solve the following problem.

\begin{prob}
Deform $\tilde{N}(t)$ to a nearby coassociative 4-fold $N(t)$, so that $N(t)\rightarrow N$ as 
$t\rightarrow 0$ in the sense of currents.
\end{prob}   
Informally, if $|\varphi|_{\tilde{N}(t)}|$  is sufficiently small we hope to make it vanish after a small perturbation.  It is clear from the observations in 
$\S$\ref{obstruct} that the conditions we imposed precisely ensure that obstructions to this procedure can be overcome.

Our problem involves deforming the \emph{non-coassociative} $\tilde{N}(t)$.
As we have seen, infinitesimal deformations of coassociative 4-folds are defined by closed self-dual 2-forms, resulting in 
a deformation theory determined by solutions to an elliptic problem.
To exploit this fact we want to be able to identify normal deformations to
 $\tilde{N}(t)$  with self-dual 2-forms, even though $\tilde{N}(t)$ is \emph{not} coassociative.  We achieve this as in \cite{Lotaydesing} by defining 
the self-dual 2-forms with respect to a different metric on $\tilde{N}(t)$ from the induced one.  This follows from 
an easy modification of \cite[Proposition 2.9 \& Lemma 4.3]{Lotaydesing}, since all one requires is that   
$\|\varphi|_{\tilde{N}(t)}\|_{C^0}$ is smaller than some universal constant, and we can in fact make this norm 
arbitrarily small by choosing $\tau$ sufficiently small since $\tilde{N}(t)\rightarrow N$.

\begin{prop}\label{metricprop}   Define $$\jmath_t:\nu(\tilde{N}(t))\rightarrow\Lambda^2T^*\tilde{N}(t)\quad\text{by}\quad
\jmath_t:v\mapsto(v\lrcorner\varphi)|_{\tilde{N}(t)}.$$  If $\tau$ is sufficiently small, there exists a unique metric $\tilde{g}(t)$
 on $\tilde{N}(t)$ such that $$\Imm\jmath_t=(\Lambda^2_+)_{\tilde{g}(t)}T^*\tilde{N}(t)\quad\text{and}\quad
 *\varphi|_{\tilde{N}(t)}=\vol_{\tilde{g}(t)}.$$  
\end{prop}

\begin{note} From now on we shall calculate all quantities on $\tilde{N}(t)$ with respect to the metric $\tilde{g}(t)$ given in Proposition \ref{metricprop}, unless stated otherwise, and we shall use the notation of Definition \ref{tildeNdfn}.
\end{note}

\subsection{Weighted spaces}

We wish to define spaces of forms on $\tilde{N}(t)$ whose behaviour on the piece we have glued into $N$ is controlled, since this is where the geometry 
is degenerating.  We achieve this using Banach spaces with weighted norms, as discussed in detail in \cite{Pacinidesing}.   To define these spaces we need an
 appropriate radius function.

\begin{dfn} Recall Definition \ref{tildeNdfn} and let $R^{\prime},\epsilon^{\prime}$ be constants so that 
$$\tau R<\tau R^{\prime}<\frac{1}{2}\tau^{\nu}<\tau^{\nu}<\epsilon^{\prime}<\epsilon.$$
We define a radius function $\rho_t:\tilde{N}(t)\rightarrow [tR,\epsilon]$ as a smooth map
such that 
$$\rho_t(x)=\left\{\begin{array}{ll} tR & x\in\chi(tK_A),\\ \epsilon & x\in \Upsilon_N(\Gamma_{\alpha_N^0(t)|_{K_N}}),\end{array}\right.$$
 and $\rho_t$ on $\Upsilon_C(\Gamma_{\alpha_C(t)})$ is a strictly increasing function of $r$ such that 
$$\rho_t\big(\Upsilon_C(r,\sigma,\alpha_C(t)(r,\sigma))\big)=r \quad\text{for $r\in[tR^{\prime},\epsilon^{\prime}]$.}$$
In other words, $\rho_t$ is a small perturbation of the piecewise smooth function which is constant on $\chi(tK_A)$ and 
$\Upsilon_N(\Gamma_{\alpha_N^0(t)|_{K_N}})$ and equal to $r$ on $\Upsilon_C(\Gamma_{\alpha_C(t)})$.  The existence of such a function $\rho_t$ on $\tilde{N}(t)$ is clear.  
\end{dfn}

\begin{dfn}\label{wspacedfn} 
Let $k\in\N$, $p\geq 1$ and $\upsilon\in\R$.  

\noindent Define $C^k_{\upsilon,t}\big(\Lambda^mT^*\tilde{N}(t)\big)$ to be the space 
 $C^k\big(\Lambda^mT^*\tilde{N}(t)\big)$ with the norm
$$\|\xi\|_{C^k_{\upsilon,t}}=\sum_{j=0}^k\sup_{\tilde{N}(t)}|\rho_t^{j-\upsilon}\nabla^j\xi|.$$
Let $L^p_{k,\upsilon,t}\big(\Lambda^mT^*\tilde{N}(t)\big)$ be the space 
 $L^p_{k}\big(\Lambda^mT^*\tilde{N}(t)\big)$ with the norm
$$\|\xi\|_{L^p_{k,\upsilon,t}}=\sum_{j=0}^k\left(\int_{\tilde{N}(t)}|\rho_t^{j-\upsilon}\nabla^j\xi|^p\rho_t^{-4}
\d\!\vol_{\tilde{g}(t)}\right)^{\frac{1}{p}}.$$
These weighted spaces are Banach spaces since the norms are equivalent to the usual norms for each fixed $t$.
\end{dfn}

These spaces have the nice feature, unlike their ``unweighted'' counterparts, that they are equivariant under dilations in $t$.  Thus, with respect to these weighted norms, we can 
understand the behaviour of quantities as $t\rightarrow 0$, which is crucial for our analysis.  For a detailed discussion of these issues see \cite{Pacinidesing}.

On CS and AC submanifolds we can define Sobolev, $C^k$ and H\"older spaces of forms, denoted $L^p_{k,\upsilon}$, $C^k_{\upsilon}$ and 
$C^{k,a}_{\upsilon}$, which depend on a weight $\upsilon\in\R$.    
Informally, these Banach spaces consist of forms whose restriction
 to any compact set lies in the usual Sobolev, $C^k$ or 
H\"older space, but which also have controlled rate of decay on the ends 
determined by $\upsilon$.  This is achieved using weighted norms as in Definition \ref{wspacedfn}, replacing $\rho_t$ by an appropriate
 radius function  -- we refer the interested reader to \cite[$\S$3.3]{Lotaydesing} or \cite{Pacinidesing} for details.  We point out that $L^2_{0,-2}=L^2$.

\subsection{Topology}

Clearly we need to know $b^2_+\big(\tilde{N}(t)\big)$ to understand deformations of $\tilde{N}(t)$.  This is a topological invariant which  
 we can determine using the topology of $\hat{N}$ and $A$. 

\begin{thm}\label{b2+thm} Using the notation of Definition \ref{H2+dfn}, \eq{Aseq} and \eq{Nseq}, we have that
\begin{align}
b^2_+\big(\tilde{N}(t)\big)= b^2_+(A)+b^2_+(\hat{N})+\dim(\Imm j^A_2\cap \Imm j^N_2).\label{b2+eq}
\end{align}
\end{thm}

\begin{proof} We can clearly choose a pair of open subsets of $\tilde{N}(t)$, diffeomorphic to $A$ and 
$\hat{N}$, which cover $\tilde{N}(t)$ and whose intersection is diffeomorphic to $C$.  By Mayer--Vietoris we then have the following exact sequence:
\begin{equation}\label{seq}
\cdots\!\longrightarrow H^m\big(\tilde{N}(t)\big)
\,{\buildrel \tilde{\iota}_m\over\longrightarrow}\,
H^m(A)\oplus H^m(\hat{N})\,{\buildrel \tilde{j}_m\over\longrightarrow}\,
H^m(\Sigma)\,{\buildrel
\tilde{\partial}_m\over\longrightarrow}\,
H^{m+1}\big(\tilde{N}(t)\big)\longrightarrow\!\cdots.
\end{equation}
Since \eq{seq} is exact we have that 
\begin{equation}\label{b2eq1}
b^2\big(\tilde{N}(t)\big)=\dim\Ker \tilde{\iota}_2+\dim\Imm\tilde{\iota}_2=\dim\Imm\tilde{\partial}_1+\dim\Ker\tilde{j}_2.
\end{equation}

We first calculate $\dim\Imm\tilde{\partial}_1$.  Using \eq{Aseq} we see that
$\dim\Imm\partial_1^A=\dim\Imm j^A_2$ 
and the same result holds for $N$ by \eq{Nseq}.
 The fact that these spaces have the same dimension is a consequence of Poincar\'e duality, which further allows us to construct 
 an isomorphism between them.  
Now $\tilde{\partial}_1$ is defined so that
 $\tilde{\partial}_1[\alpha_{\Sigma}]$ can be simultaneously viewed as $\partial_1^A[\alpha_{\Sigma}]\in H^1_{\text{cs}}(A)$ and 
 $\partial_1^N[\alpha_{\Sigma}]\in H^1_{\text{cs}}(\hat{N})$.  Hence, $\Imm\tilde{\partial}_1$ is dual to 
the intersection of $\Imm j^A_2$ and $\Imm j^N_2$.   We conclude that 
\begin{align}\label{dimeq1}
\dim\Imm\tilde{\partial}_1= \dim(\Imm j^A_2\cap\Imm j^N_2).
\end{align}

We now determine $\dim\Ker\tilde{j}_2$.  By definition, $\tilde{j}_2=j^A_2-j^N_2$, so may also calculate, recalling Definition \ref{H2+dfn}, 
\begin{align}
\dim\Ker\tilde{j}_2&=\dim\Ker j^A_2+\dim\Ker j^N_2+\dim(\Imm j^A_2\cap\Imm j^N_2)\nonumber\\
&=\dim\mathcal{J}(A)+\dim\mathcal{J}(\hat{N})+\dim(\Imm j^A_2\cap\Imm j^N_2).\label{dimeq2}
\end{align}

We deduce from \eq{b2eq1}, \eq{dimeq1} and \eq{dimeq2} that 
\begin{equation*}
b^2\big(\tilde{N}(t)\big)=\dim\mathcal{J}(A)+\dim\mathcal{J}(\hat{N})+2\dim(\Imm j^A_2\cap\Imm j^N_2).
\end{equation*}

There is no obstruction to elements of $\Imm j^N_2\subseteq H^2(\Sigma)$ lifting to closed 2-forms on $\hat{N}$ which are either 
self-dual or anti-self-dual by Proposition \ref{alphaSigmaprop}, and a similar result holds on $A$ (as noted after Theorem \ref{ACdefthm}).  We conclude that 
\begin{equation*}
b^2_{\pm}\big(\tilde{N}(t)\big)= b^2_{\pm}(A)+b^2_{\pm}(\hat{N})+\dim(\Imm j^A_2\cap\Imm j^N_2)
\end{equation*}
and deduce \eq{b2+eq}.
\end{proof}

\begin{remarks}
From the proof of Theorem \ref{b2+thm} we deduce a special case of the Novikov additivity theorem, namely that
$$b^2_+\big(\tilde{N}(t)\big)-b^2_-\big(\tilde{N}(t)\big)=b^2_+(A)-b^2_-(A)+b^2_+(\hat{N})-b^2_-(\hat{N}).$$
This is unsurprising since our proof essentially follows the argument  in \cite{AS3} for proving the Novikov additivity theorem.
\end{remarks}

\subsection{Stability}\label{stabsec}

Theorem \ref{ACdefthm} allows us to conclude that infinitesimal deformations of $C$ which 
are homogeneous of rate $(-2,0)$ always extend to genuine deformations of $A$, so the deformation 
theory of $A$ is unobstructed.   This unobstructedness follows from the fact that, for $\lambda_+\geq\lambda$ such that 
$\lambda_+\in(-2,0)\setminus\mathcal{D}$, 
$$\d\big( L^2_{4,\lambda_+}(\Lambda^2T^*A)\big)= \d\big(L^2_{4,\lambda_+}(\Lambda^2_+T^*A)\big).$$

  In contrast on $\hat{N}$ we find 
that the images of $\d$ on 2-forms and self-dual 2-forms in $L^2_{4,\mu}$ can differ.
  
\begin{dfn}\label{obstructdfn}
By the work in \cite[$\S$6]{LotayCS} we have that if $\mu_+\notin\mathcal{D}$ then there exists a finite-dimensional subspace 
 $\mathcal{O}(N,\mu_+)$ of $L^2_{3,\mu_+-1}(\Lambda^3T^*\hat{N})$ such that 
$$\d\big(L^2_{4,\mu_+}(\Lambda^2T^*\hat{N})\big)=\d\big(L^2_{4,\mu_+}(\Lambda^2_+T^*\hat{N})\big)\oplus\mathcal{O}(N,\mu_+).
$$
We call $\mathcal{O}(N,\mu)$ the \emph{obstruction space} since one of the main results in \cite{LotayCS} states that if $\mathcal{O}(N,\mu)=\{0\}$ 
then $N$ has a smooth moduli space of deformations as a CS coassociative 4-fold; that is, its deformation theory is \emph{unobstructed}.
\end{dfn}  

We shall now show that the obstruction space corresponds to closed self-dual 2-forms on $C$ which do \emph{not} lift to $\hat{N}$.  
From our matching condition in Definition \ref{matchingdfn}, we see that these are exactly the sort of obstructions we need to overcome in order to 
solve our gluing problem.  This allows us to draw a direct link between obstructions to the smoothing of $N$ and obstructions to deformations of $N$.  

We begin with the following result from \cite{LotayCS}.

\begin{prop}\label{kercsprop} Recall Definition \ref{Ddfn},   
let $\mu_0$ be the least element of $\big((-2,0)\cap\mathcal{D}\big)\cup\{0\}$ and let $\mu_-\in(-2,\mu_0)$.  
If $\mu_+\in(-2,0)\setminus\mathcal{D}$ with $\mu_+>\mu_-$ then the dimension of the kernel of $\d$ in $L^2_{4,\mu_+}(\Lambda^2_+T^*\hat{N})$ is
$$b^2_+(\hat{N})+\dim\Ker(\d^*_++\d)_{\mu_+}-\dim\Ker(\d^*_++\d)_{\mu_-}-
\!\!\!\!\!\!\sum_{\upsilon\in(-2,\mu_+)}\!\!\!\!\!\!d_{\mathcal{D}}(\upsilon),$$
where $(\d^*_++\d)_{\upsilon}$ acts on $L^2_{4,-3-\upsilon}(\Lambda^3T^*\hat{N})$.
\end{prop}

The reason for the appearance of $\Ker(\d^*_++\d)_{\upsilon}$ is that it is isomorphic to the cokernel of the map 
\begin{equation}\label{lindefmap}
(\d_++\d^*)_{\upsilon}
:L^2_{4,\upsilon}(\Lambda^2_+T^*\hat{N}\oplus\Lambda^4T^*\hat{N})\rightarrow L^2_{3,\upsilon-1}(\Lambda^3T^*\hat{N}).
\end{equation}
We need to relate $\Ker(\d^*_++\d)_{\upsilon}$ to the space of closed and coclosed 3-forms, since $\mathcal{O}(N,\mu_+)$ is isomorphic 
to the subspace of $\Ker(\d^*_++\d)_{\mu_+}$ which is orthogonal to these forms by the work in \cite{LotayCS}.  We begin with the following
observation.

\begin{lem}\label{closedlem}
If $\gamma\in L^2_{4,-1}(\Lambda^3T^*\hat{N})$ and $\d^*_+\gamma=0$ then $\d^*\gamma=0$.
\end{lem}

\begin{proof}
Since $\d^*\gamma\in L^2_{3,-2}\hookrightarrow L^2$ we can calculate 
$$\|\d^*\gamma\|_{L^2}=-\int_{\hat{N}}\d^*\gamma\w \d^*\gamma=\int_{\hat{N}}\d(*\gamma\w\d*\gamma)=0,$$
using the fact that $\d^*\gamma$ is anti-self-dual and the decay properties of $\gamma$ to ensure the integration by parts is valid.
\end{proof}

It follows from Lemma \ref{closedlem} that, for $\upsilon\leq -2$, $\Ker(\d^*_++\d)_{\upsilon}$ is equal to
\begin{equation}\label{H3eq}
\mathcal{H}^3_{\upsilon}(\hat{N})=\{\gamma\in L^2_{4,-3-\upsilon}(\Lambda^3T^*\hat{N})\,:\,\d\gamma=\d^*\gamma=0\}.
\end{equation}
Thus, the obstruction space $\mathcal{O}(N,\upsilon)=\{0\}$ if $\upsilon\leq -2$.  However, we need to calculate $\dim\mathcal{O}(N,\mu_+)$ for $\mu_+\in(-2,0)$ so 
we need to compare $\Ker(\d^*_++\d)_{\mu_+}$ and $\mathcal{H}^3_{\mu_+}$ for these rates.  This is made easier by the following fact.

\begin{prop}\label{cccsprop}
 The space $\mathcal{H}^3_{\mu_+}(\hat{N})$ in \eq{H3eq} 
is the same for all $\mu_+\in(-2,0)$.
\end{prop}

\begin{proof}
By the work in \cite{LockhartMcOwen}, changes in $\mathcal{H}^3_{\mu_+}(\hat{N})$ are governed by homogeneous closed and 
coclosed 3-forms $\gamma_{\infty}$ on $C$ of rate $-3-\upsilon$ for $\upsilon\in(-2,0)$.  We write 
\begin{equation}\label{cceq0}
\gamma_{\infty}=r^{-3-\upsilon}(r^3\beta_{\Sigma}+r^{2}\d r\w\alpha_{\Sigma})
\end{equation}
for forms $\alpha_{\Sigma},\beta_\Sigma$ on $\Sigma$.  The condition that $\d\gamma_{\infty}=\d^*\gamma_{\infty}=0$ is equivalent to
\begin{equation}\label{cceq1}
\d\alpha_{\Sigma}=-\upsilon\beta_{\Sigma},\quad\d\!*_{\Sigma}\!\alpha_{\Sigma}=0\quad\text{and}\quad \d\!*_\Sigma\!\beta_{\Sigma}=(\upsilon+2)*_\Sigma\!\alpha_\Sigma.
\end{equation}
We deduce that 
\begin{equation*}
\Delta_{\Sigma}\beta_{\Sigma}=\upsilon(\upsilon+2)\beta_{\Sigma}.
\end{equation*}
Thus $\beta_{\Sigma}=0$ if $\upsilon\in(-2,0)$ and hence by \eq{cceq1} $\alpha_{\Sigma}=0$ as well.  
\end{proof}

As previously mentioned, the work in \cite{LotayCS} shows that 
$$\dim\Ker(\d^*_++\d)_{\upsilon}=\dim\mathcal{H}^3_{\upsilon}(\hat{N})+\dim\mathcal{O}(N,\upsilon).$$
Using the notation of Proposition \ref{kercsprop}, applying Proposition \ref{cccsprop} gives:
$$\dim\Ker(\d^*_++\d)_{\mu_+}-\dim\Ker(\d^*_++\d)_{\mu_-}=\dim\mathcal{O}(N,\mu_+)-\dim\mathcal{O}(N,\mu_-).$$
As we saw, $\mathcal{O}(N,\upsilon)=\{0\}$ for $\upsilon\leq -2$.  Since $\mathcal{O}(N,\mu_-)=\mathcal{O}(N,-2+\varepsilon)$ for arbitrarily small $\varepsilon>0$, to finish we need to calculate how the obstruction space changes as the rate crosses $-2$.  The next lemma states that it does not change.

\begin{lem}\label{mu-lem} Using the notation of Proposition \ref{kercsprop} and \eq{H3eq}, 
$$\Ker(\d^*_++\d)_{\mu_-}=\mathcal{H}^3_{\mu_-}(\hat{N})\quad\text{or, equivalently,}\quad\mathcal{O}(N,\mu_-)=\{0\}.$$
\end{lem}

\begin{proof}
A form $\gamma$ adds to $\Ker(\d^*_++\d)_{\mu_-}$ at rate $-2$ 
if and only if it is asymptotic to a 3-form $\gamma_{\infty}$ on $C$ of rate $-3-(-2)=-1$ which satisfies
 $\d\gamma_{\infty}=\d^*_+\gamma_{\infty}=0$ by the work in \cite{LockhartMcOwen}.  Hence, if we abuse notation and identify $\gamma_{\infty}$ with its 
pull-back to $\hat{N}$, then $\gamma-\gamma_{\infty}\in\Ker(\d^*_++\d)_{-2}$.  Since $\Ker(\d^*_++\d)_{-2}=\mathcal{H}^3_{-2}(\hat{N})$ by Lemma \ref{closedlem},
$\gamma-\gamma_{\infty}$ is closed and coclosed, so if we show that $\d\gamma_{\infty}=\d^*\gamma_{\infty}=0$ then 
$\gamma\in\mathcal{H}^3_{\mu_-}(\hat{N})$ as we desire.

 Writing $\gamma_{\infty}$ as in \eq{cceq0} for $\upsilon=-2$, we see that $\alpha_{\Sigma}$ and $\beta_{\Sigma}$ satisfy
\begin{equation*}
\d\alpha_{\Sigma}=2\beta_\Sigma\quad\text{and}\quad *_\Sigma\!\d\!*_\Sigma\!\alpha_{\Sigma}+\d\!*_\Sigma\!\beta_\Sigma=0.
\end{equation*}
Taking $\d^*_{\Sigma}$ of the second equation, 
we deduce that $\beta_\Sigma$ is harmonic and exact, so must be zero.  Thus $\alpha_{\Sigma}$ is a closed and coclosed 2-form.  
By \eq{cceq1}, this means that $\d\gamma_{\infty}=\d^*\gamma_{\infty}=0$ as required.
\end{proof}

Combining the results in this section we deduce the following.

\begin{prop}\label{kerOprop}
If $\mu_+\in(-2,0)\setminus\mathcal{D}$ the dimension of the kernel of $\d$ in $L^2_{4,\mu_+}(\Lambda^2_+T^*\hat{N})$ is
$$b^2_+(\hat{N})+\dim\mathcal{O}(N,\mu_+)-
\!\!\!\!\!\!\sum_{\upsilon\in(-2,\mu_+)}\!\!\!\!\!\!d_{\mathcal{D}}(\upsilon).$$
\end{prop}

From the theory of \cite{LockhartMcOwen}, as we cross rate $\upsilon\in(-2,\mu_+)$ the dimension of the kernel of $\d$ changes 
if and only if there exists a closed self-dual 2-form 
$\alpha$ on $\hat{N}$ which is asymptotic to a homogeneous closed self-dual 2-form $\alpha_C$ on $C$ of rate $\upsilon$, in the sense that, for some $\delta>0$,
$$|\nabla^j_C(\Psi_N^*\alpha-\alpha_C)|=O(r^{\upsilon+\delta-j})\text{ as }r\rightarrow 0\text{ for all }j\in\N.$$
 Recall that we say that $\alpha_C$ lifts to $\alpha$ on $\hat{N}$ for convenience.  Notice that such $\alpha$ is of order $O(r^{\upsilon})$ and so cannot lie in 
$L^2_{4,\mu_+}$.  Hence $\alpha$ is subtracted from the kernel of $\d$ as we cross $\upsilon$ (or, if you prefer, added as we decrease the rate 
below $\upsilon$).  So the kernel will decrease by the maximal amount as we cross rate $\upsilon$ if and only if every homogeneous closed self-dual 2-form on $C$
of rate $\upsilon$ lifts to $\hat{N}$.

Proposition \ref{kerOprop} thus allows us to conclude that every homogeneous closed self-dual 2-form on $C$  of rate 
$\upsilon\in(-2,\mu_+)$ lifts to a closed self-dual 2-form on  $\hat{N}$ if and only if $\mathcal{O}(N,\mu_+)=\{0\}$.  In particular,
 $\alpha_A^0$ given in Proposition \ref{alphaAprop} will lift to $\hat{N}$ as in
 the matching condition in Definition \ref{matchingdfn} if $j_2^A[\alpha_A^0]\in j_2^N\big(H^2(\hat{N})\big)$  and $\mathcal{O}(N,\lambda)=\{0\}$.  
 We can therefore check that 
the matching condition is satisfied purely using a topological criterion and the spectrum of the curl operator on $\Sigma$.
The analytic part of the matching condition therefore relates to the work in \cite{Lotaystab} on \emph{stability} of coassociative conical singularities.

We now recall the notion of \emph{stability index} for a coassociative cone from \cite{Lotaystab}.

\begin{dfn}\label{stabdfn}
Let $\mathcal{C}$ denote a deformation family of coassociative cones in $\R^7$ containing $C$ which is closed under the action of translations and $\GG_2$ transformations. 
The \emph{$\mathcal{C}$-stability index} of $C$ is
$$\ind_{\mathcal{C}}(C)=\sum_{\upsilon\in(-1,1]}\!\!\!\! d_{\mathcal{D}}(\upsilon)-\dim\mathcal{C}.$$
If the family $\mathcal{C}$ consists solely of the $\GG_2\ltimes\R^7$
 transformations of $C$, we simply write $\ind_{\mathcal{C}}(C)=\ind(C)$ and call $\ind(C)$ the \emph{stability 
index} of $C$.  

Since translations of $C$ trivially provide coassociative deformations of $C$ of order $O(1)$, they define 
 homogeneous closed self-dual 2-forms on $C$ of rate $0$, and thus $d_{\mathcal{D}}(0)$ is at least equal to the dimension of the space of translations of $C$.  
 Moreover, $d_{\mathcal{D}}(1)$ is equal to the dimension of  infinitesimal coassociative conical deformations of $C$.  Overall, we have that 
 $\ind_{\mathcal{C}}(C)\geq 0$ and equals zero if and only if $\dim\mathcal{C}=d_{\mathcal{D}}(0)+d_{\mathcal{D}}(1)$ and $d_{\mathcal{D}}(\upsilon)=0$ for 
 all $\upsilon\in(-1,1)\setminus\{0\}$.

We say that $C$ is $\mathcal{C}$-\emph{stable} (or \emph{stable}) if $\ind_{\mathcal{C}}(C)=0$ (or $\ind(C)=0$).
\end{dfn}

The $\mathcal{C}$-stability index is a non-negative integer invariant and it follows from \cite[Proposition 4.11]{Lotaystab} that 
the deformation theory of $N$ as a CS coassociative 4-fold, where we allow the singularity to move in $M$ and the cone at the singularity to deform in 
$\mathcal{C}$, is unobstructed if $\ind_{\mathcal{C}}(C)=0$.  In particular, we have the following result.

\begin{prop} If\/ $\ind_{\mathcal{C}}(C)=0$ then 
$\dim\mathcal{O}(N,\mu_+)=0$ for all $\mu_+\in(-2,0)$, using the notation of Definitions \ref{obstructdfn} and \ref{stabdfn}. 
\end{prop}

\noindent We deduce the following.

\begin{prop}\label{stabprop}  Recall Definitions \ref{matchingdfn}, \ref{topmatchingdfn}, \ref{obstructdfn} and \ref{stabdfn}.
\begin{itemize}
\item[\emph{(a)}] If the topological matching condition holds and $\mathcal{O}(N,\lambda)=\{0\}$, then the matching condition between $A$ and $N$ holds.
\item[\emph{(b)}] 
If $C$ is $\mathcal{C}$-stable for some deformation family $\mathcal{C}$, then the matching condition between $A$ and $N$ is 
equivalent to the topological matching condition.
\end{itemize}
\end{prop}

\begin{remark}
Proposition \ref{mainthm3} follows from Proposition \ref{stabprop}(b).
\end{remark}

The reader might wonder why the sum in the stability index is over $(-1,1]$ rather than the seemingly more natural $(-2,1]$, since $-2$ is the key rate at which 
obstructions begin to appear.  However, $\upsilon=-1$ is the crucial rate for understanding the 3-forms in the cokernel of $(\d_++\d^*)_{\upsilon}$ given in 
\eq{lindefmap}, i.e. $\d^*_++\d$ acting on 
$L^2_{4,-3-\upsilon}$ which is $L^2_{4,-2}$ when $\upsilon=-1$.  This cokernel determines the obstruction space for the deformations of coassociative 4-folds with conical singularities.  It follows from 
these observations, as can be seen in \cite{Lotaystab}, that any potential obstructions for the deformation theory of 
 CS coassociative 4-folds arising from cokernel forms for rates in $(-2,-1]$ are ineffective.  Hence, obstructions can only appear from cokernel forms for rates 
$\upsilon>-1$ and this leads to the given formula for the stability index.  

The author speculates that one can improve the definition of the stability index 
still further by showing that obstructions remain ineffective for higher rates, potentially for all $\upsilon<0$ 
(and the same could be true in the special Lagrangian setting).  
This may ultimately mean that the purely analytic parts of the 
matching condition are always satisfied; i.e.~that the only obstructions to gluing $tA$ and $N$ are topological, which is a natural criterion to expect.

\section{Desingularization: analysis}

In this section we apply analytic techniques to prove our main result (Theorem \ref{mainthm1}).  
We begin by deriving a key Sobolev embedding inequality, 
 which involves the construction of an ``approximate kernel'' for the exterior derivative on self-dual 2-forms on our glued manifold.    
  We then view our desingularization problem as a fixed point problem for a certain map, so we show that this map is a contraction 
using further analytic estimates.  To derive the embedding inequality and the estimates we shall make crucial use of the geometric 
preliminaries of $\S$\ref{geometry}.  

For the whole of this section we let $\delta$ satisfy
\begin{equation}\label{deltaeq}
0\leq\delta<\text{max}\{-(1+2\lambda),1\}
\end{equation}
  and be such that $[-2-\delta,-2+\delta]\cap\mathcal{D}=\{-2\}$ if $-2\in\mathcal{D}$
 and let $\delta=0$ if $-2\notin\mathcal{D}$.  This is possible by our assumption that $\lambda<-\frac{1}{2}$ 
and the properties of $\mathcal{D}$.

\subsection{The approximate kernel}  

Here we obtain our Sobolev embedding inequality, closely following the work in \cite{Lotaydesing} with improvements in light of Pacini's work in \cite{Pacinidesing}.  The result is a bound (depending on $t$ in an explicit way) 
for the norm of self-dual 2-forms $\alpha$ on $\tilde{N}(t)$, transverse to the closed forms, by the norm of $\d\alpha$.  Since we have scale-invariant 
Sobolev embedding inequalities on $A$ and $\hat{N}$, the idea is to use closed self-dual 2-forms on $A$ and $\hat{N}$ to
 build a subspace of the self-dual 2-forms on $\tilde{N}(t)$ 
 which ``approximates'' the kernel of the exterior derivative on self-dual 2-forms. 

If the closed self-dual 2-forms on $A$ and $\hat{N}$ decay sufficiently fast on the ends, we can simply cut them off and approximate them using a compactly supported 2-form which can easily be viewed as an 
approximate kernel form on $\tilde{N}(t)$.  However, at the critical decay rate, namely at the $L^2$ growth rate $-2$, this cut off procedure will not work and so 
one needs to define approximate kernel forms by interpolating between $L^2$ kernel forms on $A$ and $L^2$ kernel forms on $\hat{N}$, when this is possible.
  Using the topological calculations in $\S$\ref{geometry}, we find that these forms define an approximate kernel of equal dimension to the actual kernel.

However, the interpolation between $L^2$ kernel forms on $A$ and $\hat{N}$ is not always possible, and one can detect this topologically by the work in $\S$\ref
{obstruct}. When this occurs, we have closed self-dual 2-forms on $A$ which do \emph{not} define approximate kernel forms on $\tilde{N}(t)$ and so give potential 
obstructions.  These forms will cause the Sobolev embedding constant to blow up as $t\rightarrow 0$ but because we can identify the forms explicitly we can 
determine the rate at which the blow up occurs.

\begin{prop}\label{AapKer}
Let $\mathcal{K}^{A}_{\pm}$ be the (finite-dimensional) kernel of 
\begin{equation*}
\d:L^{2}_{4,-2\pm\delta}(\Lambda^2_+T^*A)\rightarrow L^{2}_{3,-3\pm\delta}(\Lambda^3T^*A)
\end{equation*}
and let $\mathcal{K}^A_{0}$ be such that $$\mathcal{K}^A_+=\mathcal{K}^A_-\oplus\mathcal{K}_{0}^A.$$
There is a subspace $\mathcal{K}_{\text{\emph{ap}}}^{A}\subseteq C^{\infty}_{\text{\emph{cs}}}(\Lambda^2_+T^*A)$, $L^2$-orthogonal to
 $\mathcal{K}_{0}^A$ with $\dim\mathcal{K}_{\text{\emph{ap}}}^{A}=\dim\mathcal{K}^A_-$, and a constant $C(A)>0$ such that if $\alpha\in L^{2}_{4,-2+\delta}\big(\Lambda^2_+T^*A\big)$ 
satisfies $\langle\alpha,\beta\rangle_{L^2}=0$ for all $\beta\in \mathcal{K}_{\text{\emph{ap}}}^{A}\oplus\mathcal{K}^A_{0}$ then 
\begin{equation}\label{Aesteq}
\|\alpha\|_{L^2_{4,-2+\delta}}\leq C(A)\|\d\alpha\|_{L^2_{3,-3+\delta}}.
\end{equation}
Moreover, this estimate holds for the same constant $C(A)$ on $tA$ for all $t>0$.
\end{prop}

\begin{proof}
The map
\begin{equation}\label{Add*eq}
\d_++\d^*:L^{2}_{4,-2\pm\delta}(\Lambda^2_+T^*A\oplus\Lambda^4T^*A)\rightarrow L^{2}_{3,-3+\delta}(\Lambda^3T^*A)
\end{equation}
is elliptic and Fredholm by choice of $\delta$ (c.f.~\cite[Proposition 5.4]{LotayAC} and the remarks preceding the statement), and therefore has a finite-dimensional kernel of smooth forms by elliptic
 regularity. Since $\mathcal{K}^{A}_{\pm}$ is contained in this kernel it is also necessarily finite-dimensional and consists of smooth
 forms.
 
We can cut off the forms in $\mathcal{K}^A_-$ appropriately at infinity to define a space
 $\mathcal{K}^{A}_{\text{ap}}$ of compactly supported self-dual 2-forms on $A$, $L^2$-orthogonal to $\mathcal{K}_{0}^A$,
 so that if $\alpha\in\mathcal{K}^A_-$ is 
$L^2$-orthogonal to $\mathcal{K}^A_{\text{ap}}$, then $\alpha=0$.  (This is a manifestation of the fact that, by definition, 
$C^{\infty}_{\text{cs}}$ is dense in $L^2_{4,-2-\delta}$.)  In other words, we can ensure that the $L^2$-orthogonal complement of
 $\mathcal{K}^A_{\text{ap}}$ in $L^2_{4,-2+\delta}$ is transverse to $\mathcal{K}^A_-$ and contains $\mathcal{K}_{0}^A$.   
The theory of elliptic operators on weighted Sobolev spaces as in \cite{Mazya} applied to \eq{Add*eq} 
allows us to deduce the existence of the constant $C(A)$ using standard techniques.  

By definition of the weighted norm,
 if $\beta$ is an $m$-form on $A$ then $\|\beta\|_{L^2_{k,\upsilon}}$ scales with order $t^{-\upsilon-m}$ under dilation by 
$t$.  Therefore both sides of \eq{Aesteq} scale by the same factor under dilation and thus we can choose $C(A)$ independent of $t$.  (See \cite{Pacinidesing} for 
a detailed discussion of the scaling properties of weighted Sobolev norms.)
\end{proof}

We can also prove the following analogue of Proposition \ref{AapKer} in a similar (easier) manner which we omit.

\begin{prop}\label{hatNapKer}
Let $\mathcal{K}^{N}$ be the (finite-dimensional) kernel of 
\begin{equation*}
\d:L^{2}_{4,-2+\delta}(\Lambda^2_+T^*\hat{N})\rightarrow L^{2}_{3,-3+\delta}(\Lambda^3T^*\hat{N}).
\end{equation*}
There is a subspace $\mathcal{K}_{\text{\emph{ap}}}^{N}\subseteq C^{\infty}_{\text{\emph{cs}}}(\Lambda^2_+T^*\hat{N})$, 
with $\dim\mathcal{K}_{\text{\emph{ap}}}^{N}=\dim\mathcal{K}^N$, and a constant $C(N)>0$ such that if $\alpha\in L^{2}_{4,-2+\delta}(\Lambda^2_+T^*\hat{N})$ 
satisfies $\langle\alpha,\beta\rangle_{L^2}=0$ for all $\beta\in \mathcal{K}_{\text{\emph{ap}}}^{N}$ then 
\begin{equation*}\label{Nesteq}
\|\alpha\|_{L^2_{4,-2+\delta}}\leq C(N)\|\d\alpha\|_{L^2_{3,-3+\delta}}.
\end{equation*}
\end{prop}

We now wish to 
define our approximate kernel.  We begin with $\mathcal{K}^A_{\text{ap}}$ and define a diffeomorphism 
\begin{equation*}
\Psi_{A,t}:t\hat{A}(t)=tK_A\cup t\Phi_A\big((R,t^{-1}\epsilon)\times\Sigma\big)\rightarrow\chi(tK_A)\cup\Upsilon_C(\Gamma_{\alpha_C(t)})
\end{equation*}
 by 
\begin{equation}\label{PsiAteq}
\Psi_{A,t}(x)=\left\{\begin{array}{ll} \chi(x) & x\in tK_A,\\
\Upsilon_C\big(r,\sigma,\alpha_C(t)(r,\sigma)\big) & x=t\Phi_A(t^{-1}r,\sigma). \end{array}\right.
\end{equation}
If $\tau$ is sufficiently small we may identify the metrics, hence the self-dual 2-forms, on $t\hat{A}(t)$ and $\Psi_{A,t}\big(t\hat{A}(t)\big)$.  This allows 
us to view the open subset $t\hat{A}(t)$ of $tA$  as a subset of the desingularization $\tilde{N}(t)$.

Let $\beta_1^A,\ldots,\beta_{m_A}^A$ be a basis for $\mathcal{K}^A_{\text{ap}}$.
Since $\nu<1$, we can choose $\tau$ so that 
$$\supp\beta_i^A\subseteq K_A\sqcup\Phi_A\big((R,\textstyle\frac{1}{2}t^{\nu-1})\times\Sigma\big)\subseteq \hat{A}(t)$$ for $i=1,\ldots,m_A$, so we may identify 
the $\beta_i^A$ with forms on $t\hat{A}(t)$ in the obvious manner (namely, via pullback of the map $t\hat{A}(t)\mapsto \hat{A}(t)$, which is 
the inverse of dilation by a factor of $t$), which we denote by the same symbols. 
Using $\Psi_{A,t}$ we can define for each $\beta_i^A$ on $t\hat{A}(t)$ a corresponding self-dual 2-form $\xi_i^A$ on 
$\tilde{N}(t)$ which vanishes outside $\Psi_{A,t}(\supp\beta_i^A)$. 

\begin{dfn}\label{apKerdfn1}
Let $\tilde{\mathcal{K}}^A_{\text{ap}}(t)=\Span\{\xi_1^A,\ldots,\xi_{m_A}^A\}$.
\end{dfn}

We now deal with $\mathcal{K}^N_{\text{ap}}$ and define a diffeomorphism 
\begin{equation*}
\Psi_{N,t}:\hat{N}(t)=\Psi_N\big((tR,\epsilon)\times\Sigma\big)\cup K_N\rightarrow \Upsilon_C(\Gamma_{\alpha_C(t)})\cup
\Upsilon_N(\Gamma_{\alpha^N_0(t)|_{K_N}})
\end{equation*}
 by 
\begin{equation}\label{PsiNteq}
\Psi_{N,t}(x)=\left\{\begin{array}{ll} \Upsilon_C\big(r,\sigma,\alpha_C(t)(r,\sigma)\big) & x=\Psi_N(r,\sigma),\\
\Upsilon_N\big(x,\alpha^N_0(t)(x)\big)  & x\in K_N. \end{array}\right.
\end{equation}
For $\tau$ sufficiently small we can identify the metrics, hence the self-dual 2-forms, on $\hat{N}(t)$ and $\Psi_{N,t}(\hat{N}(t))$.  As above, this 
allows us to view the open subset $\hat{N}(t)$ of $\hat{N}$ as a subset of $\tilde{N}(t)$.
 
Let $\beta_{1}^N,\ldots,\beta^N_{m_N}$ be a basis for $\mathcal{K}^N_{\text{ap}}$.  Since $\nu>0$ we can ensure, by making $\tau$ smaller if necessary, that 
 $$\supp\beta_i^N\subseteq K_N\cup\Psi_N\big((t^{\nu},\epsilon)\times\Sigma\big)\subseteq\hat{N}(t)$$ for all $i$.    Using $\Psi_{N,t}$ we can then 
 define for each $\beta_i^N$ a corresponding self-dual 2-form $\xi_i^N$ on $\tilde{N}(t)$ which vanishes outside $\Psi_{N,t}(\supp\beta_i^N)$.  

\begin{dfn}\label{apKerdfn2}
Let $\tilde{\mathcal{K}}^N_{\text{ap}}(t)=\Span\{\xi_1^N,\ldots,\xi_{m_N}^N\}$.
\end{dfn}

It will be important to identify the closed self-dual 2-forms on $A$ which extend to $\hat{N}$ and those which do not, as we saw from our discussion of the 
obstructions to the gluing problem in $\S$\ref{obstruct}.  Motivated by Proposition
\ref{alphaSigmaprop} we can split $\mathcal{K}_0^A$ in the following useful way.

\begin{dfn}\label{apKerdfnIO} Recall \eq{Aseq} and \eq{Nseq}.
Let $$\mathcal{K}^{\mathcal{I}}=\{\alpha\in\mathcal{K}^A_0\,:\,j_2^A[\alpha]\in \Imm j_2^N\}$$
and let $\mathcal{K}^{\mathcal{O}}$ be such that $\mathcal{K}^A_0=\mathcal{K}^{\mathcal{I}}\oplus\mathcal{K}^{\mathcal{O}}$.
\end{dfn}

The notation $\mathcal{I}$ and $\mathcal{O}$ reflects the fact that elements in $\mathcal{K}^{\mathcal{I}}$ extend to infinitesimal
 deformations of $\hat{N}$ and hence $\tilde{N}(t)$, whereas $\mathcal{K}^\mathcal{O}$ gives potential obstructions.
Let $\beta^{\mathcal{I}}_1,\ldots,\beta^{\mathcal{I}}_{m_{\mathcal{I}}}$ and 
$\beta^{\mathcal{O}}_1,\ldots,\beta^{\mathcal{O}}_{m_{\mathcal{O}}}$ form bases of $\mathcal{K}^{\mathcal{I}}$ and 
$\mathcal{K}^{\mathcal{O}}$.  Since every element of $\Imm j_2^A$ lifts to  a form in $\mathcal{K}^A_0$ 
by the work in \cite{LotayAC}, we have that 
\begin{equation}\label{dimKIeq}
\dim\mathcal{K}^{\mathcal{I}}=\dim(\Imm j_2^A\cap\Imm j_2^N).
\end{equation}

We shall now define a part of the approximate kernel using $\mathcal{K}^{\mathcal{I}}$ and we explain the idea.  
Since each $\beta\in\mathcal{K}^{\mathcal{I}}$ is asymptotic to a closed
self-dual 2-form $\beta^C$ on $C$ whose cohomology class lies in the image of $j^N_2:H^2(\hat{N})\rightarrow H^2(\Sigma)\cong H^2(C)$, we can 
find a closed self-dual 2-form $\gamma$ on $\hat{N}$ which is also asymptotic to $\beta^C$.  We then interpolate between $\beta$ and $\gamma$ to define 
a self-dual 2-form on $\tilde{N}(t)$ which is ``almost'' closed, i.e.~it is closed except on some small compact set.  

By the general theory in \cite{LockhartMcOwen}, as discussed in the particular case of interest in \cite{LotayAC}, for
 each $i=1,\ldots,m_{\mathcal{I}}$ there exists a closed self-dual 2-form $\beta^{C}_{i}$ on $C$, which is homogeneous of order $O(r^{-2})$, such
 that, for some $\epsilon_0>0$,
$$\big|\nabla^j_C\big(\Phi_A^*\beta^\mathcal{I}_i(r,\sigma)-\beta^C_i(r,\sigma)\big)\big|=O(r^{-2-\epsilon_0-j})\quad\text{as }r\rightarrow\infty\text{ for all }j\in\N.$$  
By Proposition \ref{alphaSigmaprop}, there exist closed self-dual 2-forms $\gamma^\mathcal{I}_i$ on 
$\hat{N}$ so that, for some other $\epsilon_0>0$, 
$$\big|\nabla^j_C\big(\Psi_N^*\gamma^\mathcal{I}_i(r,\sigma)-\beta^C_i(r,\sigma)\big)\big|=O(r^{-2+\epsilon_0-j})\quad\text{as }r\rightarrow0\text{ for all }j\in\N.$$

Define a diffeomorphism $\Psi_{C,t}:(tR,\epsilon)\times\Sigma\rightarrow \Upsilon_C(\Gamma_{\alpha_C(t)})$ by 
\begin{equation}\label{PsiCteq}
\Psi_{C,t}(r,\sigma)=\Upsilon_C\big(r,\sigma,\alpha_C(t)(r,\sigma)\big).
\end{equation}
We can now define a self-dual 2-form $\xi_i^{\mathcal{I}}$ on $\tilde{N}(t)$, for $i=1,\ldots,m_{\mathcal{I}}$, so that on $\chi(tK_A)$ it equals
 $\beta^{\mathcal{I}}_i$ (using the identification $\Psi_{A,t}$), on $\Upsilon_N(\Gamma_{t\alpha_N^0|_{K_N}})$ it equals $\gamma^{\mathcal{I}}_i$ (using the identification $\Psi_{N,t}$), and on $\Upsilon_C(\Gamma_{\alpha_C(t)})$ it interpolates between 
 these definitions in the following way:
$$\Psi_{C,t}^*\xi_i^{\mathcal{I}}(r,\sigma)=\big(1-f_{\text{inc}}(2t^{-\nu}r-1)\big)\Psi_{A,t}^*\beta^\mathcal{I}_i(r,\sigma)+
f_{\text{inc}}(2t^{-\nu}r-1)\Psi_{N,t}^*\gamma^{\mathcal{I}}_i(r,\sigma),$$
where $f_{\text{inc}}$ is given in Definition \ref{tildeNdfn}.  
Notice that $\xi_i^{\mathcal{I}}$ is closed except on the region where $tA$ is connected to $\hat{N}$ to form $\tilde{N}(t)$.

\medskip

We now turn to $\beta\in\mathcal{K}^{\mathcal{O}}$ which do \emph{not} define approximate kernel forms.  We cut off $\beta$
to define a self-dual 2-form $\xi$ on $\tilde{N}(t)$ which vanishes on $\hat{N}(t)$ and is closed except on a small compact set so that, 
as $t\rightarrow 0$, $\xi$ converges back to $\beta$ on $A$, after re-scaling.  Although $\xi$ will again be ``almost'' closed, this time there is 
no corresponding closed self-dual 2-form on $\tilde{N}(t)$ which it approximates.  Moreover, since $\xi$ converges to a non-trivial closed form as 
$t\rightarrow0$ it is clear that $\xi$ will contribute to the blow up of the Sobolev embedding constant as $t\rightarrow 0$.

 Define $f^{\mathcal{O}}: A\rightarrow [0,1]$ to be a smooth function such that
$$f^{\mathcal{O}}(x)=1\text{ for }x\in K_A,\quad\supp\d f^{\mathcal{O}}\subseteq \Phi_A\big((\textstyle\frac{1}{2}t^{\nu-1},t^{\nu-1})\times\Sigma\big)$$ and $f^{\mathcal{O}}$ is decreasing in $r$ 
on $\Imm\Phi_A$.  
For $i=1,\ldots,m_{\mathcal{O}}$ there exists a smooth self-dual 2-form
 $\xi^{\mathcal{O}}_i$ on $\tilde{N}(t)$ such that
$$\Psi_{A,t}^*(\xi^{\mathcal{O}}_i)=\beta^{\mathcal{O}}_i\text{ on }tK_A\cup t\Phi_A\big((R,\textstyle\frac{1}{2}t^{\nu-1})\times\Sigma\big)$$
and
$$\supp\xi^{\mathcal{O}}_i\subseteq \Psi_{A,t}\big(tK_A\sqcup t\Phi_A\big((R,t^{\nu})\times\Sigma\big)\big),$$
again using the identifications as before.  We can achieve this essentially by cutting off the form
 $\beta^{\mathcal{O}}_i$ using $f^{\mathcal{O}}$.  Notice that $\xi^\mathcal{O}_i$ is closed except on the interpolation region between 
$tA$ and $N$ in $\tilde{N}(t)$. 

\begin{dfn}\label{KIKOdfn} 
Let $\tilde{\mathcal{K}}^{\mathcal{I}}(t)=\{\xi_1^\mathcal{I},\ldots,\xi_{m_\mathcal{I}}^{\mathcal{I}}\}$ and let 
$\tilde{\mathcal{K}}^{\mathcal{O}}(t)=\{\xi_1^\mathcal{O},\ldots,\xi_{m_\mathcal{O}}^{\mathcal{O}}\}$.
\end{dfn}

Combining Definitions \ref{apKerdfn1}, \ref{apKerdfn2} and \ref{KIKOdfn} leads to our approximate kernel.

\begin{dfn}\label{apKerdfn}
Let $\tilde{\mathcal{K}}_{\text{ap}}(t)=\tilde{\mathcal{K}}_{\text{ap}}^A(t)\oplus\tilde{\mathcal{K}}_{\text{ap}}^N(t)
\oplus\tilde{\mathcal{K}}^{\mathcal{I}}(t)$.  
\end{dfn}

\noindent Observe that, by construction, the sums in $\tilde{\mathcal{K}}_{\text{ap}}(t)$ are direct and therefore that 
$$\dim\tilde{\mathcal{K}}_{\text{ap}}(t)=b^2_+(A)+b^2_+(\hat{N})+\dim(\Imm j^A_2\cap \Imm j^N_2)=b^2_+\big(\tilde{N}(t)\big),$$ 
using \eq{dimKIeq} and Theorem \ref{b2+thm}.  

We may now state our Sobolev embedding inequality.

\begin{thm}\label{apKerthm}  Recall $\nu$ given in \eq{nueq}. 
There is a constant $C(\tilde{N})$, independent of $t$, such that if $\alpha\in L^2_{4}\big(\Lambda^2_+T^*\tilde{N}(t)\big)$ satisfies
$\langle\alpha,\xi\rangle_{L^2}=0$ for all $\xi\in\tilde{K}_{\text{\emph{ap}}}(t)$ then
$$\|\alpha\|_{L^2_{4,-2+\delta,t}}\leq C(\tilde{N})t^{-\delta(1-\nu)}\|\d\alpha\|_{L^2_{3,-3+\delta,t}}.$$
\end{thm}

\begin{note}
If we could choose $\delta=0$, i.e.~if $-2\notin\mathcal{D}$ which is equivalent to $b^1(\Sigma)=0$, then we have a uniform 
Sobolev embedding constant independent of $t$.  This tallies with the work in \cite{Lotaydesing}.
\end{note}

\begin{proof}
The idea of the proof is first to show that the only way that the Sobolev embedding constant can blow up as $t\rightarrow 0$ is if we have a sequence converging 
to a closed self-dual 2-form on $A$ or $N$.  The construction of $\tilde{\mathcal{K}}_{\text{ap}}(t)$ means that the only non-trivial limit that can occur is 
an element of $\mathcal{K}^{\mathcal{O}}$.  We can therefore just study the behaviour of forms in 
$\tilde{\mathcal{K}}^{\mathcal{O}}(t)$ as $t\rightarrow 0$ to deduce our estimate. 

Suppose, for a contradication, that there exists a decreasing sequence of positive numbers $t_n\rightarrow 0$ 
and a sequence $$\alpha_n\in L^2_{4,-2+\delta,t_n}\big(\Lambda^2_+T^*\tilde{N}(t_n)\big)$$ such that 
 $\langle\alpha_n,\xi\rangle_{L^2}=0$ for all $\xi\in\tilde{\mathcal{K}}_{\text{ap}}(t_n)$ and 
\begin{equation}\label{falphaeq7}
\|\alpha_n\|_{L^2_{4,-2+\delta,t_n}}=1\geq nt_n^{-\delta(1-\nu)}\|\d\alpha_n\|_{L^2_{3,-3+\delta,t_n}}.
\end{equation}
Therefore, the sequences $$\Psi_{A,t_n}^*\alpha_n\in L^2_{4,-2+\delta}(\Lambda^2_+T^*A)\quad\text{and}\quad
\Psi_{N,t_n}^*\alpha_n\in L^2_{4,-2+\delta}(\Lambda^2_+T^*\hat{N})$$ are bounded.  So, by the compact embedding theorem for weighted Sobolev spaces \cite[Theorem 4.9]{Lockhart}, after passing to a subsequence, both
 sequences converge in $L^2_{3,-2+\delta^{\prime}}$, to $\xi^A$ and $\xi^N$ say, where $\delta^{\prime}>\delta$ for $\xi^A$ and $\delta^{\prime}<\delta$ for 
$\xi^N$.  
 
Using the bounds for $\|\d\alpha_n\|_{L^2_{3,-3+\delta,t_n}}$ we see that 
$$\|\d(\Psi_{A,t_n}^*\alpha_n)\|_{L^2_{2,-3+\delta^{\prime}}}\leq \|\d(\Psi_{A,t_n}^*\alpha_n)\|_{L^2_{3,-3+\delta}}\rightarrow 0\quad\text{as }n\rightarrow\infty.$$  Hence $\d\xi^A=0$ and similarly $\d\xi^N=0$.  
Elliptic regularity implies that $\xi^A$ and $\xi^N$ are smooth and, as $-2+\delta\notin\mathcal{D}$, the work in \cite{LockhartMcOwen} shows that the space of closed self-dual 2-forms 
is the same at rates $-2+\delta$ and $-2+\delta^{\prime}$ if $\delta^{\prime}$ is sufficiently close to $\delta$.  We conclude that $\Psi_{A,t_n}^*\alpha_n$ and $\Psi_{N,t_n}^*\alpha_n$ converge in $L^2_{4,-2+\delta}$ to $\xi^A$ and $\xi^N$ respectively.    

The fact that 
$\alpha_n$ is $L^2$-orthogonal to $\tilde{\mathcal{K}}_{\text{ap}}(t_n)$ means that $\xi^A$ is $L^2$-orthogonal to 
$\mathcal{K}_{\text{ap}}^A\oplus\mathcal{K}^{\mathcal{I}}$ and $\xi^N$ is $L^2$-orthogonal to $\mathcal{K}_{\text{ap}}^N$.  
From Proposition \ref{hatNapKer} we deduce that $\xi^N=0$.  If $\xi^A$ is $L^2$-orthogonal to $\mathcal{K}^{\mathcal{O}}$ 
then $\xi^A=0$ by Proposition \ref{AapKer}, so we must have that $\xi^A\in\mathcal{K}^{\mathcal{O}}$.  

Overall 
$\alpha_n$ is a sequence of forms such that $\Psi_{N,t_n}^*\alpha_n\rightarrow 0$ and $\Psi_{A,t_n}^*\alpha_n\rightarrow
\xi^A\in\mathcal{K}^{\mathcal{O}}$.  Thus, for all $n$ sufficiently large, $\alpha_n$ is well approximated by elements in 
$\tilde{\mathcal{K}}^{\mathcal{O}}(t_n)$, so we now analyse these forms.

By definition, elements of $\mathcal{K}^{\mathcal{O}}$ 
are kernel forms on $A$ which do \emph{not} extend to corresponding kernel forms on $N$ and thus do not define kernel forms on $\tilde{N}(t)$.  Hence,
 $\tilde{\mathcal{K}}^{\mathcal{O}}(t)$ must be transverse to the closed self-dual 2-forms on $\tilde{N}(t)$ for $\tau$ small.  
Therefore there exist some ($t$-dependent) constants $C_t(\tilde{N})>0$ such that, for all $\alpha\in\tilde{\mathcal{K}}^{\mathcal{O}}(t)$, 
$$\|\alpha\|_{L^2_{4,-2+\delta,t}}\leq C_t(\tilde{N})\|\d\alpha\|_{L^2_{3,-3+\delta,t}}.$$

Recall Definition \ref{KIKOdfn} and the discussion preceding it.  Any $\alpha\in\tilde{\mathcal{K}}^{\mathcal{O}}(t)$ is 
identified with $f^{\mathcal{O}}\beta$ for some 
$\beta\in\mathcal{K}^{\mathcal{O}}$.  Using the definition of $f^{\mathcal{O}}$, the facts that $\d\beta=0$ and satisfies $|\nabla^j_C\Phi_A^*\beta|=O(r^{-2-j})$ for all $j\in\N$ as $r\rightarrow\infty$, together with the assumption that $\delta<1$ in \eq{deltaeq} we calculate
\begin{align*}
\|\d(f^{\mathcal{O}}\beta)\|^2_{L^2_{3,-3+\delta}}&=\|\d f^{\mathcal{O}}\w\beta\|^2_{L^2_{3,-3+\delta}}\\
&=\sum_{j=0}^3\int_{\Phi_A\big((\frac{1}{2}t^{\nu-1},t^{\nu-1})\times\Sigma\big)}|r^{j+3-\delta}\nabla^j(\d f^{\mathcal{O}}\w\beta)|^2r^{-4}\d\vol_{g_0|_A}\\
&=O\left(\int_{\frac{1}{2}t^{\nu-1}}^{t^{\nu-1}}|r^{3-\delta}t^{1-\nu}r^{-2}|^2 r^{-1}\d r\right)\\
&=O\big(t^{2\delta(1-\nu)}\big).
\end{align*}
Notice that this norm tends to zero as $t\rightarrow 0$ as we would expect.  We deduce that 
$$\|\alpha\|_{L^2_{4,-2+\delta,t}}\leq Ct^{-\delta(1-\nu)}\|\d\alpha\|_{L^2_{3,-3+\delta,t}}$$
for some constant $C>0$.

Hence, for $n>>1$, there exists some other constant $C>0$ so that
$$\|\alpha_n\|_{L^2_{4,-2+\delta,t_n}}=1\leq Ct_n^{-\delta(1-\nu)}\|\d\alpha_n\|_{L^2_{3,-3+\delta,t_n}}.$$
This contradicts \eq{falphaeq7}.
\end{proof}

Using our estimate we have the following crucial result.

\begin{thm}\label{invthm} Recall Definition \ref{apKerdfn}.  The exterior derivative
\begin{align*}
\d:\tilde{\mathcal{K}}_{\text{\emph{ap}}}(t)^{\perp}\subseteq L^2_{4,-2+\delta,t}&\big(\Lambda^2_+T^*\tilde{N}(t)\big)\\
&\longrightarrow \big\{\xi\in L^2_{3,-3+\delta,t}\big(\Lambda^3T^*\tilde{N}(t)\big):\xi\,\text{is exact}\big\}
\end{align*}
is a bounded invertible linear map between Banach spaces with bounded linear inverse $P_{\d}$ satisfying 
$\|P_{\d}\|\leq C(\tilde{N})t^{-\delta(1-\nu)}$.
\end{thm}

\begin{proof}  Recall that $\Imm\d|_{\Lambda^2_+}=\Imm\d|_{\Lambda^2}$  
on compact Riemannian 4-manifolds (c.f.~\cite[Proposition 2.10]{Lotaydesing}).  An immediate consequence of Theorem \ref{apKerthm} is that the $L^2$-orthogonal projection of $\mathcal{H}^2_+\big(\tilde{N}(t)\big)$, given in Definition \ref{H2+dfn},  
to $\tilde{\mathcal{K}}_{\text{ap}}(t)$ is injective.  It is also surjective since the dimensions of the two spaces are equal by Theorem \ref{b2+thm}.  

 The domain and range of $\d$ given are clearly Banach spaces and the existence of the 
bounded inverse $P_{\d}$ is now clear.  The bound on the operator norm of $P_{\d}$ is simply a restatement of the estimate in Theorem \ref{apKerthm}.
\end{proof}

\subsection{The contraction map}

Recall the tubular neighbourhood constructions for $C$, $A$ and $\hat{N}$ in Propositions \ref{Cnbdprop}-\ref{NAnbdprop} which identified nearby deformations 
with graphs of self-dual 2-forms.  Using the isomorphism
 $\nu(\tilde{N}(t))\cong\Lambda^2_+T^*\tilde{N}(t)$ given by Proposition \ref{metricprop}, a straightforward adaptation of the 
work in \cite[$\S$5.1-5.2]{Lotaydesing} shows that we can construct a tubular neighbourhood $\tilde{T}(t)$ of $\tilde{N}(t)$ in $M$ which is identified with the
 $\tilde{\epsilon}$-ball about the zero section in $C^1_{1,t}\big(\Lambda^2_+T^*\tilde{N}(t)\big)$ for some $\tilde{\epsilon}>0$, in a manner which is compatible 
with the constructions in Propositions \ref{Cnbdprop}-\ref{NAnbdprop}.  Rather than repeating the details here we refer the interested reader to 
\cite[$\S$5.1-5.2]{Lotaydesing}.  We can use this construction to describe coassociative deformations of $\tilde{N}(t)$. 

\begin{dfn}\label{defmapdfn} Let $\alpha\in C^1\big(\Lambda^2_+T^*\tilde{N}(t)\big)$ with
 $\|\alpha\|_{C^1_{1,t}}<\tilde{\epsilon}$.  Using the construction 
discussed above, we can define a nearby deformation $\tilde{N}_{\alpha}(t)\subseteq \tilde{T}(t)$ of 
$\tilde{N}(t)$ with a natural diffeomorphism $f_{\alpha}(t):\tilde{N}(t)\rightarrow \tilde{N}_{\alpha}(t)$.  Let
\begin{equation*}
F_t(\alpha)=f_{\alpha}(t)^*\left(\varphi|_{\tilde{N}_{\alpha}(t)}\right).
\end{equation*}
By the coassociativity of $N$ and $A$, $F_t(\alpha)$ is exact.  
\end{dfn}

\noindent By construction, the zeros of $F_t$
 correspond exactly to nearby coassociative deformations of $\tilde{N}(t)$.

As in \cite[Proposition 6.2]{Lotaydesing}, we can say more about the deformation map $F_t$.

\begin{prop}\label{defmapprop}
For $\alpha\in C^1\big(\Lambda^2_+T^*\tilde{N}(t)\big)$ with
 $\|\alpha\|_{C^1_{1,t}}<\tilde{\epsilon}$, we may write
\begin{equation*}
F_t(\alpha)=\varphi|_{\tilde{N}(t)}+\d\alpha +Q_{t}(\alpha)
\end{equation*}
for a smooth map $Q_t$ depending on $\alpha$ and $\nabla\alpha$.  
Moreover, if $\alpha\in L^{2}_{4}\big(\Lambda^2_+T^*\tilde{N}(t)\big)$ with $\|\alpha\|_{C^1_{1,t}}<\tilde{\epsilon}$,
 then $$Q_{t}(\alpha)\in \d\big(L^2_{4}\big(\Lambda^2_+T^*\tilde{N}(t)\big)\big) \subseteq L^{2}_{3}\big(\Lambda^3T^*\tilde{N}(t)\big).$$ 
\end{prop}

From Proposition \ref{defmapprop} we see that solving $F_t(\alpha)=0$ is equivalent to solving 
\begin{equation*}
\d\alpha=-\varphi|_{\tilde{N}(t)}-Q_t(\alpha).
\end{equation*}  
Since the right-hand side lies in $\d\big(L^2_{4}\big(\Lambda^2_+T^*\tilde{N}(t)\big)\big)$, we can use Theorem \ref{invthm}  and try to solve 
\begin{equation}\label{cmap1}
\alpha=P_{\d}\big(-\varphi|_{\tilde{N}(t)}-Q_t(\alpha)\big)
\end{equation}
for $\alpha\in\tilde{\mathcal{K}}_{\text{ap}}(t)^{\perp}\subseteq L^2_4\big(\Lambda^2_+T^*\tilde{N}(t)\big)$.  The idea is to apply the
 Contraction Mapping Theorem to give a solution to \eq{cmap1}, which will in turn define a zero of $F_t$ and hence a 
 coassociative deformation of $\tilde{N}(t)$.

\begin{dfn}\label{ctdfn}
For $\alpha\in\tilde{\mathcal{K}}_{\text{ap}}(t)^{\perp}\subseteq L^2_4\big(\Lambda^2_+T^*\tilde{N}(t)\big)$ with
 $\|\alpha\|_{C^1_{1,t}}<\tilde{\epsilon}$, define
\begin{equation*}
\mathcal{C}_t(\alpha)=P_{\d}\big(-\varphi|_{\tilde{N}(t)}-Q_t(\alpha)\big)\in \tilde{\mathcal{K}}_{\text{ap}}(t)^{\perp}\subseteq L^2_4\big(\Lambda^2_+T^*\tilde{N}(t)\big).
\end{equation*}
As observed, fixed points of $\mathcal{C}_t$ define elements of $\Ker F_t$.  Moreover, given a fixed point $\alpha$ of 
$\mathcal{C}_t$, we may apply the Implicit Function Theorem and parameterise the elements of $\Ker F_t$ near $\alpha$ by 
$\tilde{\mathcal{K}}_{\text{ap}}(t) \cong \mathcal{H}^2_+\big(\tilde{N}(t)\big)$.
\end{dfn}

Given the estimate on the norm of $P_{\d}$ in Theorem \ref{invthm}, to show that $\mathcal{C}_t$ is a contraction on some neighbourhood of zero in $L^2_{4,-2+\delta,t}$, 
it is enough to obtain estimates on the $L^2_{3,-3+\delta,t}$ norm of $\varphi|_{\tilde{N}(t)}$ and $Q_t(\alpha)-Q_t(\beta)$ for $\alpha,\beta\in L^2_4$.

We begin with the estimate on the norm of $\varphi|_{\tilde{N}(t)}$.

\begin{prop}\label{phiestprop} There exists a constant $C(\varphi)>0$, independent of $t$, such that
$$\|\varphi|_{\tilde{N}(t)}\|_{L^2_{3,-3+\delta,t}}\leq C(\varphi)t^{\nu(3-\delta)+2(1-\lambda)(1-\nu)}.$$
\end{prop}

\begin{remark}
It is in the proof of this proposition that we finally use the constraint on $\nu$ in \eq{nueq}.
\end{remark}

The key idea in the proof is that our matching condition and the assumption that $\lambda<-\frac{1}{2}$ ensure that the terms which should naively give the 
largest contribution to $|\varphi|_{\tilde{N}(t)}|$ in fact are zero or effectively cancel.

\begin{proof} 
For convenience, we use the diffeomorphisms given in \eq{PsiAteq}, \eq{PsiNteq} and \eq{PsiCteq} to decompose $\tilde{N}(t)$ into three pieces:
\begin{gather*}
\tilde{N}_l(t)=
\Psi_{A,t}\left(tK_A\cup t\Phi_{A}\big((R,\textstyle\frac{1}{2}\displaystyle t^{\nu-1})\times\Sigma\big)\right)\subseteq
\chi\big(t\hat{A}(t)\big),
\\
\tilde{N}_m(t)=\Psi_{C,t}\big([\textstyle\frac{1}{2}\displaystyle t^{\nu},t^{\nu}]\times\Sigma\big)\,\;\text{and}\;\
\tilde{N}_u(t)=\Psi_{N,t}\big(\Psi_N\big((t^{\nu},\epsilon)\times\Sigma\big)\cup K_N\big).
\end{gather*} 
First observe trivially that 
\begin{equation}\label{phiest1}
\|\varphi|_{\tilde{N}(t)}\|_{L^2_{3,-3+\delta,t}}^2=\|\varphi\|_{L^2_{3,-3+\delta,t}\big(\tilde{N}_l(t)\big)}^2+
\|\varphi\|_{L^2_{3,-3+\delta,t}\big(\tilde{N}_m(t)\big)}^2+\|\varphi\|_{L^2_{3,-3+\delta,t}\big(\tilde{N}_u(t)\big)}^2.
\end{equation}
  
We begin with estimating the norm of $\varphi$ on $\tilde{N}_l(t)$.  Since $\chi^*(\varphi)$ agrees with $\varphi_0$ at $0$, we have that 
$\chi^*(\varphi)=\varphi_0+O(r)$ on $B(0;\epsilon_M)$.  In fact, $\chi^*(\nabla^j\varphi)=\nabla^j\varphi_0+O(r^{1-j})$ 
for $j\in\N$.  Since $\varphi_0|_{t\hat{A}(t)}\equiv 0$, we have that 
$$|r^{j+3-\delta}\chi^*(\nabla^j\varphi)|=O(r^{4-\delta})\quad\text{on $t\hat{A}(t)$.}$$
Hence the dominate terms in calculating the norm of $\varphi$ on $\tilde{N}_l(t)$ arise on $\chi(t\hat{A}(t))\setminus\chi(tK_A)$.
 We may calculate
\begin{align*}
\sum_{j=0}^3\int_{t\Phi_A\big((R,\frac{1}{2}t^{\nu-1})\times\Sigma\big)}
|r^{j+3-\delta}\chi^*(\nabla^j\varphi)|^2r^{-4}\d\!\vol_{g_0|_{tA}}
&=O\left(\int_{tR}^{\frac{1}{2}t^{\nu}}\!\!r^{2(4-\delta)}r^{-1} \d r\right)\\&=O(t^{2 \nu(4-\delta)}).
\end{align*}
Hence, there exists a $t$-independent constant $C_l>0$ such that
\begin{equation}\label{phiest2}
\|\varphi\|_{L^2_{3,-3+\delta,t}\big(\tilde{N}_l(t)\big)}^2\leq C_l t^{2\nu(4-\delta)}.
\end{equation}

For $p\in\tilde{N}_u(t)$, because $\varphi|_{\hat{N}}\equiv 0$, we may decompose $\varphi(p)$ in a similar manner to \eq{phiexpandeq}: 
$$\varphi(p)=\d\alpha_N^0(t)(p)+P_N\big(p,\alpha_N^0(t)(p),\nabla\alpha_N^0(t)(p)\big).$$
Since $\alpha_N^0(t)=\sum_{i=1}^dt^{1-\lambda_i}\alpha_N^i$ is closed and $P_N(p,\alpha(p),\nabla\alpha(p))$ is dominated by $|r^{-1}\Psi_N^*\alpha(r,\sigma)|^2$ 
and $|\nabla_C\Psi_N^*\alpha(r,\sigma)|^2$, we see that the largest contribution to $|\varphi|_{\tilde{N}_u(t)}|$ arises from 
$\Psi_{N,t}\big(\Psi_N\big((t^{\nu},\epsilon)\times\Sigma\big)\big)$.  From our matching condition in Definition \ref{matchingdfn} we have that 
$|\nabla^j_C\Psi_N^*\alpha_N^i|=O(r^{\lambda_i-j})$, so on $\tilde{N}_u(t)$ we have $|\varphi|=O(\sum_{i=1}^dt^{2(1-\lambda_i)}r^{2\lambda_i-2})$.    We calculate
\begin{align*}
\int_{\Psi_N((t^{\nu},\epsilon)\times\Sigma)}|r^{3-\delta}\sum_{i=1}^dt^{2(1-\lambda_i)}r^{2\lambda_i-2}|^2 &r^{-4}\d\vol_{g_{\varphi}|_{\hat{N}}}\\
&=O\left(\sum_{i=1}^d t^{4(1-\lambda_i)}\int_{t^\nu}^\epsilon r^{4\lambda_i+1-2\delta}\d r\right)\\
&=O\left(\sum_{i=1}^dt^{\nu(4\lambda_i+2-2\delta)+4(1-\lambda_i)}\right).
\end{align*}
Since $\lambda_i\leq\lambda$ for all $i$, we see that 
there exists a $t$-independent constant $C_u>0$ such that
\begin{equation}\label{phiest3}
\|\varphi\|_{L^2_{3,-3+\delta,t}\big(\tilde{N}_u(t)\big)}^2\leq C_u t^{2\nu(3-\delta)+4(1-\lambda)(1-\nu)}.
\end{equation}

We are now left with $\tilde{N}_m(t)$, which will be the key contribution to calculate.  Since the graph of $\alpha_N$ over $C$ defines $\hat{N}$ near $z$, 
we can view $\tilde{N}_m(t)$ as the graph of $\beta=\alpha_C(t)-\alpha_N$ over $N$ via Proposition \ref{NAnbdprop}.  Since $N$ is coassociative and we can approximate the metric 
on $\tilde{N}_m(t)$ using the conical metric, we see that 
 $$|\varphi|_{\tilde{N}_m(t)}|\leq c |\d\beta+P_C(\beta,\nabla_C\beta)|$$
 for some $c>0$ independent of $t$.  Now we can use \eq{alphaCeq} to decompose $\beta$ into four terms, which we can estimate using the fact that $2\lambda-1<-2$ in Proposition \ref{alphaAprop} 
and the matching condition  as follows:
\begin{align}
|t^3\nabla_C^j\delta_{t^{-1}}^*\alpha_A^0(r,\sigma)|&=O\left(\sum_{i=1}^d t^3.t^{-j-2}(t^{-1}r)^{\lambda_i-j}\right)\nonumber\\&=O(t^{1-\lambda}r^{\lambda-j}),\label{betaest1}\\
|t^3\nabla_C^j\delta_{t^{-1}}^*\alpha_A^{\prime}(r,\sigma)|&=O(t^{1-\lambda_-}r^{\lambda_--j}),\label{betaest2}\\
\big|t^{1-\lambda_i}\nabla_C^j\big(\Psi_N^*\alpha_N^i(r,\sigma)-\alpha_A^i(r,\sigma)\big)\big|&=O(t^{1-\lambda_i}r^{\lambda_i+\delta_0-j})\nonumber\\&=O(t^{1-\lambda}
r^{\lambda+\delta_0-j}),\label{betaest3}\\
|\nabla_C^j\alpha_N(r,\sigma)|&=O(r^{\mu-j}),\label{betaest4}
\end{align}
where $\lambda_-<-2$.  
Overall, $|r^{-1}\beta(r,\sigma)|^2$ and $|\nabla_C\beta(r,\sigma)|^2$ are dominated by terms of order $O((t^{-1}r)^{2(\lambda-1)})$ and 
$O(r^{2(\mu-1)})$.  The choice of $\nu$ in \eq{nueq} ensures that $\nu>(1-\lambda)/(\mu-\lambda)$ since $\mu>1+\delta_0$ by assumption \eq{delta0eq}, hence $(1-\lambda)(1-\nu)<\nu(\mu-1)$ so that
terms with the former exponent give the greatest contribution as 
$t\rightarrow 0$ on $\tilde{N}_m(t)$ (recalling that $r=O(t^{\nu})$).  Thus, we see from \eq{alphaCeq} that 
\begin{equation}\label{betaest5}
\big|P_C\big(r,\sigma,\beta(r,\sigma),\nabla_C\beta(r,\sigma)\big)\big|=O\big((t^{-1}r)^{2(\lambda-1)}\big).
\end{equation}

Recall that $\alpha_A^0$, $\alpha_N^i$ and $\alpha_A^i$ are all closed forms, so we have from \eq{alphaCeq} that 
\begin{align*}
&\d\beta(r,\sigma)\\ &= t^3\big(1-f_{\text{inc}}(2t^{-\nu}r-1)\big)\delta_{t^{-1}}^*\d\alpha_A^{\prime}(r,\sigma)-2t^{3-\nu}\d f_{\text{inc}}(2t^{-\nu}r-1)\w\delta_{t^{-1}}^*\alpha_A^{\prime}(r,\sigma)\\
&+2t^{-\nu}\d f_{\text{inc}}(2t^{-\nu}r-1)\w\sum_{i=1}^d t^{1-\lambda_i}\big(\Psi_N^*\alpha_N^i(r,\sigma)-\alpha_A^i(r,\sigma)\big)\\
&+\big(f_\text{inc}(2t^{-\nu}r-1)-1\big)\d\alpha_N(r,\sigma)+2t^{-\nu}\d f_{\text{inc}}(2t^{-\nu}r-1)\w\alpha_N(r,\sigma).
\end{align*}
From \eq{betaest1}-\eq{betaest4}, we may calculate
\begin{align*}
|\d\beta(r,\sigma)|\leq c^{\prime}(t^{1-\lambda_-}r^{\lambda_--1}+t^{1-\nu-\lambda_-}r^{\lambda_-}+t^{1-\nu-\lambda}r^{\lambda+\delta_0}+r^{\mu-1}+t^{-\nu}r^{\mu})
\end{align*}
for some $c^{\prime}>0$ independent of $r$, $\sigma$ and $t$.  The dominant terms are therefore of order $O(t^{1-\nu-\lambda}r^{\lambda+\delta_0})$ and 
$O(r^{\mu-1})$.  Again using \eq{nueq} and recalling that $r=O(t^{\nu})$ we find that the former terms dominate since $\mu-\delta_0>1+\delta_0$ by \eq{delta0eq} so $1-\lambda+\nu(\lambda+\delta_0-1)<\nu(\mu-1)$.
Thus we have that
\begin{equation}\label{betaest6}
|\d\beta(r,\sigma)|=O(t^{1-\nu-\lambda}r^{\lambda+\delta_0}).
\end{equation}

Finally, we can compare  \eq{betaest5} and \eq{betaest6}.  We see by \eq{nueq} that $2(1-\lambda)(1-\nu)<1-\lambda+\nu(\lambda+\delta_0-1)$.  Thus 
the terms in \eq{betaest5} are dominant as $t\rightarrow 0$, which is crucial for our later argument.  We can therefore estimate
$$\int_{\frac{1}{2}t^{\nu}}^{t^{\nu}}|r^{3-\delta}(t^{-1}r)^{2(\lambda-1)}|^2r^{-1}\d r=O\big(t^{2\nu(3-\delta)+4(1-\lambda)(1-\nu)}\big).$$
Thus there exists a $t$-independent constant $C_m>0$ such that
\begin{equation}\label{phiest4}
\|\varphi\|_{L^2_{3,-3+\delta,t}\big(\tilde{N}_m(t)\big)}^2\leq C_m t^{2\nu(3-\delta)+4(1-\lambda)(1-\nu)}.
\end{equation} 
Combining \eq{phiest1}, \eq{phiest2}, \eq{phiest3} and \eq{phiest4} gives the result.
\end{proof}

Following \cite[Proposition 6.3]{Lotaydesing} we can estimate the norm of $Q_t$.

\begin{prop}\label{Qestprop}
There exists a constant $C(Q)$, independent of $t$, such that if $\alpha,\beta\in L^{2}_{4}\big(\Lambda^2_+T^*\tilde{N}(t)\big)$ with $\|\alpha\|_{C^1_{1,t}},\|\beta\|_{C^1_{1,t}}<\tilde{\epsilon}$ then
\begin{align}
\|Q_t(\alpha)&-Q_t(\beta)\|_{L^{2}_{3,-3+\delta,t}}\nonumber\\
&\leq C(Q) t^{-3+\delta}\|\alpha-\beta\|_{L^{2}_{4,-2+\delta,t}}\big(\|\alpha\|_{L^2_{4,-2+\delta,t}}+
\|\beta\|_{L^2_{4,-2+\delta,t}}\big).\label{Qesteq}
\end{align}
\end{prop}

Before proving this we have the following lemma, which explains the appearance of the factor of $t^{-3+\delta}$ in Proposition \ref{Qestprop} as 
the $t$-dependence of the Sobolev embedding constant between weighted Sobolev spaces of different weights.

\begin{lem}\label{embedlem}
Let $\alpha\in L^2_4\big(\Lambda^2_+T^*\tilde{N}(t)\big)$.  Then $\alpha\in C^1\big(\Lambda^2_+T^*\tilde{N}(t)\big)$ and 
there exists a constant $c>0$, independent of $\alpha$ and $t$, such that
$$\|\alpha\|_{C^1_{1,t}}\leq c t^{-3+\delta}\|\alpha\|_{L^2_{4,-2+\delta,t}}.$$
\end{lem}

\begin{proof}
The Sobolev Embedding Theorem gives a
continuous embedding $L^2_4\hookrightarrow C^1$.  Examination of the definition of the weighted norms shows
 there exists a $t$-independent constant $c_0$ such that, for all $\alpha\in L^2_4$,
\begin{equation}\label{embedest0} 
\|\alpha\|_{C^1_{-2+\delta,t}}\leq c_0\|\alpha\|_{L^2_{4,-2+\delta,t}};
\end{equation}
i.e.~the embedding constant $L^2_{4,-2+\delta,t}\hookrightarrow C^1_{-2+\delta,t}$ is independent of $t$. 

We now calculate 
\begin{align}
\|\alpha\|_{C^1_{1,t}}=\sup(|\rho_t^{-1}\alpha|+|\nabla\alpha|)&=
\sup(|\rho_t^{-3+\delta}\rho_t^{2-\delta}\alpha|+|\rho_t^{-3+\delta}\rho_t^{3-\delta}\nabla\alpha|)\nonumber\\
&\leq c_1 t^{-3+\delta}\|\alpha\|_{C^1_{-2+\delta,t}}\label{embedest1}
\end{align}
 for some constant $c_1$ independent of $t$ and $\alpha$.  Combining \eq{embedest0} and \eq{embedest1} proves the lemma.
\end{proof}

\begin{proof}[Proposition \ref{Qestprop}]
In the proof of \cite[Proposition 6.2]{Lotaydesing} and following \cite[Proposition 6.3]{Lotaydesing}, it is explained that 
\begin{equation}\label{embedesteq}
|Q_t(\alpha)-Q_t(\beta)|=O\Big(\big(|\rho_t^{-1}(\alpha-\beta)|+|\nabla(\alpha-\beta)|\big)\big(
|\rho_t^{-1}\alpha|+|\nabla\alpha|+|\rho_t^{-1}\beta|+|\nabla\beta|\big)\Big);
\end{equation}
that is, $Q_t$ is dominated by quadratic terms in $\rho_t^{-1}\alpha$ and $\nabla\alpha$ when
 $\|\alpha\|_{C^1_{1,t}}<\tilde{\epsilon}$.  Therefore, an inequality of the type \eq{Qesteq} must hold for some (possibly 
$t$-dependent) constant $C(Q)$ (see, for example, \cite[Proposition 5.8]{JoyceCS3} for a detailed description of the type of 
argument involved).  It suffices therefore to show that $C(Q)$ can be chosen to be independent of $t$.

We may calculate using \eq{embedesteq} with $\beta=0$ to show that
\begin{align}
\|Q_t(\alpha)\|^2_{L^2_{0,-3+\delta,t}} &=\int_{\tilde{N}(t)}\rho_t^{2(3-\delta)}|Q_t(\alpha)|^2\rho_t^{-4}\d\!\vol_{\tilde{g}(t)}\nonumber\\
&=O\Big(\|\alpha\|_{C^1_{1,t}}^2\int_{\tilde{N}(t)}\rho_t^{2(2-\delta)}\rho_t^2(|\rho_t^{-1}\alpha|+|\nabla\alpha|)^2\rho_t^{-4}\d\!\vol_{\tilde{g}(t)}\Big)\nonumber\\
&=O(\|\alpha\|_{C^1_{1,t}}^2\|\alpha\|_{L^2_{1,-2+\delta,t}}^2).\label{embedest2}
\end{align}
Using Lemma \ref{embedlem} and \eq{embedest2} shows that 
\begin{equation}\label{embedest3}
\|Q_t(\alpha)\|_{L^2_{0,-3+\delta,t}}\leq c_2t^{-3+\delta}\|\alpha\|_{L^2_{4,-2+\delta,t}}^2
\end{equation}
for some $t$-independent constant $c_2$.  We deduce that \eq{embedest3} can be improved to give \eq{Qesteq} with constant $C(Q)$ 
independent of $t$.
%
%
\end{proof}

We may now show that $\mathcal{C}_t$ is indeed a contraction.

\begin{thm}\label{fpthm}
Let 
$$3\nu-\delta+2(1-\lambda)(1-\nu)>\kappa>3\nu-\delta+(3+\delta)(1-\nu)=3-\nu\delta$$
 (which is possible since $\nu<1$ and $2(1-\lambda)>3+\delta$ by \eq{deltaeq}) and let 
$$B_{t^\kappa}=\{\alpha\in\tilde{\mathcal{K}}_{\text{\emph{ap}}}(t)^{\perp}\subseteq L^2_4\big(\Lambda^2_+T^*\tilde{N}(t)\big)
:\|\alpha\|_{L^2_{4,-2+\delta,t}}\leq t^\kappa\}.$$
Then $\mathcal{C}_t:B_{t^\kappa}\rightarrow B_{t^\kappa}$ has a unique fixed point $\tilde{\alpha}(t)$.
\end{thm}

\begin{proof}
Let $\alpha,\beta\in B_{t^\kappa}$.  By Lemma \ref{embedlem}
$$\|\alpha\|_{C^1_{1,t}}\leq ct^{-3+\delta+\kappa},$$
 so we first choose $\tau$ such that $c\tau^{-3+\delta+\kappa}<\tilde{\epsilon}$ so that $F_t(\alpha)$, 
and thus $\mathcal{C}_t(\alpha)$, is well-defined.  This is possible by choice of $\kappa>3-\nu\delta>3-\delta$. 

Using Theorem \ref{invthm}, Proposition \ref{phiestprop} and Proposition \ref{Qestprop}, we calculate
\begin{align*}
\|\mathcal{C}_t&(\alpha)\|_{L^2_{4,-2+\delta,t}}\\
&=\big\|P_{\d}\big(-\varphi|_{\tilde{N}(t)}-Q_t(\alpha)\big)\big\|_{L^2_{4,-2+\delta,t}}\\
&\leq C(\tilde{N})t^{-\delta(1-\nu)}(\|\varphi|_{\tilde{N}(t)}\|_{L^2_{3,-3+\delta,t}}+\|Q_t(\alpha)\|_{L^2_{3,-3+\delta,t}})\\
&\leq C(\tilde{N})t^{-\delta(1-\nu)}\big(C(\varphi)t^{\nu(3-\delta)+2(1-\lambda)(1-\nu)}+C(Q)t^{-3+\delta}\|\alpha\|_{L^2_{4,-2+\delta,t}}^2\big).
\end{align*}
Taking $\tau$ such that 
$$C(\tilde{N})\big(C(\varphi)\tau^{3\nu-\delta+2(1-\lambda)(1-\nu)-\kappa}+C(Q)\tau^{-3+\nu\delta+\kappa}\big)<1,$$
which is possible by the choice of $\kappa$, ensures that $\mathcal{C}_t(\alpha)\in B_{t^\kappa}$.

Using Theorem \ref{invthm} and Proposition \ref{Qestprop} again, we deduce that
\begin{align*}
\|\mathcal{C}_t(\alpha)&-\mathcal{C}_t(\beta)\|_{L^2_{4,-2+\delta,t}}\\&=\big\|P_{\d}\big(Q_t(\alpha)-Q_t(\beta)\big)\big\|_{L^2_{4,-2+\delta,t}}\\
&\leq C(\tilde{N})C(Q)t^{-3+\nu\delta}\|\alpha-\beta\|_{L^2_{4,-2+\delta,t}}(\|\alpha\|_{L^2_{4,-2+\delta,t}}+\|\beta\|_{L^2_{4,-2+\delta,t}}).
\end{align*}
We finally take $\tau$ such that
$$2C(\tilde{N})C(Q)\tau^{-3+\nu\delta+\kappa}<1$$
so that $\mathcal{C}_t:B_{t^{\kappa}}\rightarrow B_{t^{\kappa}}$ is a contraction.  Applying the Contraction Mapping Theorem 
 gives the result.
\end{proof}

Our main result (Theorem \ref{mainthm1}) now follows from the next theorem.

\begin{thm}\label{desingthm} For all $t\in(0,\tau)$, let $N(t)=\tilde{N}_{\tilde{\alpha}(t)}(t)$ as in Definition \ref{defmapdfn} with $\tilde{\alpha}(t)$ given by Theorem \ref{fpthm}.  Then $N(t)$ is a smooth compact coassociative 4-fold 
such that $N(t)\rightarrow N$ as $t\rightarrow 0$ in the sense of currents.
\end{thm}

\begin{proof}
Since $L^2_4\hookrightarrow C^{1,a}$ by the Sobolev Embedding Theorem, we can apply the method of proof of 
\cite[Proposition 7.16]{Lotaydesing} to show that $\tilde{\alpha}(t)$ is smooth.  The result is now immediate by definition of $\tilde{N}(t)$.
\end{proof}

We can now deduce Corollary \ref{mainthm2} 
since we can 
parameterize the zeros of $F_t$ near $\tilde{\alpha}(t)$ using closed self-dual 2-forms on $\tilde{N}(t)$.  

\begin{prop}\label{countthm} There is a smooth family of compact coassociative smoothings of $N$ of dimension
$b^2_+(A)+b^2_+(\hat{N})+\dim(\Imm j_2^A\cap\Imm j_2^N)$.
\end{prop}

\section{Applications}

In this section we give some applications of our main results.  We first use these results to describe the relationship between the moduli space of ``matching pairs'' of AC and CS coassociative 
4-folds which can 
be used in our desingularization and the moduli
 space of the smooth compact coassociative 4-fold we construct.  This work leads us to deduce Proposition \ref{mainthm4} which
gives evidence, in the stable case, for local surjectivity of our gluing; i.e.~that all nearby smooth coassociative 4-folds to the given CS coassociative 4-fold arise 
from our desingularization method.  We then discuss 
 examples where our theory applies and consequences.

\subsection{Moduli spaces}

We now describe how the moduli spaces of the CS and AC building blocks ``fit together'' with the moduli space of smoothings, in 
a similar manner to \cite[$\S$8]{JoyceCS5}.

Suppose we have an almost $\GG_2$ manifold $M$ and a matching pair of a CS coassociative 4-fold $N\subseteq M$ and an AC coassociative 4-fold $A\subseteq\R^7$ with asymptotic cone $C\cong\R^+\times\Sigma$ 
to which Theorem \ref{mainthm1} applies. 
Hence we have $\tau>0$ and smooth compact coassociative 4-folds $N(t)$ for $t\in(0,\tau)$ such that $N(t)\rightarrow N$ as $t\rightarrow 0$.

Notice that we are free to re-scale $A$ and maintain both its AC convergence to $C$ and the matching condition with $N$.  Moreover, since $t\in(0,\tau)$ determines the scale of $A$ used in the gluing to construct $N(t)$, namely that we form the connect sum of $tA$ and $N$, if we choose $sA$ instead of $A$ in the 
desingularization, then we are allowed now to use $t\in(0,\frac{\tau}{s})$.  In other words, we can vary $\tau$ by re-scaling our initial choice of $A$ so, in some sense, there is not a natural scale in the problem.
  To remedy this, we can make 
$\tau$ canonical for our initial choice $A$ by taking the supremum, which will be finite.  
Using an appropriate dilation we may re-scale $A$ such that 
$\tau=2$.   All of the $N(t)$ are diffeomorphic to the same compact coassociative 4-fold, so we set $X=N(1)$ for definiteness.  In this way, we have fixed the 
scale.

  We make a series of definitions of the moduli spaces for convenience. 

\begin{dfn}\label{moduliNdfn}  Let $\mathcal{M}(N)$ denote the moduli space of CS coassociative deformations of $N$ with cone $C$ and rate $\mu_0$, where 
$\mu_0\leq\mu$ with $(1,\mu_0]\cap\mathcal{D}=\emptyset$.  The CS coassociative 4-folds $N^{\prime}$ in $\mathcal{M}(N)$ are deformations of $N$ which have 
the same cone $C$ and rate $\mu_0$ at their conical singularity $z'$ but $z'$ is not require to coincide with the singularity $z$ of $N$.  The deformation theory of $N$ can be obstructed by the work in \cite{LotayCS}, so $\mathcal{M}(N)$ is not a manifold in general.
\end{dfn}

Let $N'\in\mathcal{M}(N)$.  Recall that we have a map 
 $j^{N'}_2:H^2(\hat{N'})\rightarrow H^2(\Sigma)$ as in \eq{Nseq}.  Moreover, $j_{N'}^2=i_{N'}^*$ where $i_{N'}:\Sigma\rightarrow N'$ is an inclusion map.  More 
 precisely, for $N$ we can take $i_N(\sigma)=\Psi_N(\frac{\epsilon}{2},\sigma)$, using the notation of Definition \ref{csdfn}.  As $N'\in\mathcal{M}(N)$, 
there exists a diffeomorphism $f:\hat{N}\rightarrow \hat{N'}$ which sends the end of $\hat{N}$ to the end of $\hat{N'}$ and preserves the convergence to the 
singular point.  Hence, we may take $i_{N'}=f\circ i_n$, which implies that $j_{N'}^2=i_{N'}^*=i_N^*\circ f^*=j_N^2\circ f^*$.  Since 
$f^*:H^2(\hat{N'})\rightarrow H^2(\hat{N})$ is an isomorphism, this means that $\Imm j_{N'}^2=\Imm j_{N}^2$.  
 We clearly have the analogous result that if $A'$ is an AC deformation of $A$ with the same cone and rate of convergence then $\Imm j_{A'}^2=\Imm j_A^2$.

Recall that for any coassociative 4-fold $A^{\prime}$ in $\R^7$, $\varphi_0$ defines an element $[\varphi_0]\in H^2(A^{\prime})\cong H^3(\R^7;A^{\prime})$.  
 By Proposition \ref{infdilprop}, we know that if $v$ is the dilation vector field on $\R^7$ and $u'$ is the normal projection of $v|_{A'}$ then 
 $\jmath_{A'}(u')=u'\lrcorner\varphi_0|_{A'}$ is closed and $[\jmath_{A'}(u')]=3[\varphi_0]\in H^2(A^{\prime})$.  Now, although $\Imm j_{A'}^2=\Imm j_A^2$, it 
 is possible that $j_{A'}^2[\varphi_0]\neq j_A^2[\varphi_0]\in H^2(\Sigma)$: for example, if $A'$ is a dilation of $A$ then $j_{A'}^2[\varphi_0]$ is 
a multiple of $j_A^2[\varphi_0]$.  This leads us to the following definition. 

\begin{dfn}\label{moduliAdfn}
Let $\mathcal{M}(A)$ denote the moduli space of AC coassociative deformations $A^{\prime}$ of $A$ with cone $C$ and rate $\lambda_0$, where 
$\lambda_0\geq\lambda$ with $[\lambda_0,-\frac{1}{2})\cap\mathcal{D}=\emptyset$, such that $j^{A^{\prime}}_2[\varphi_0]\in \Imm j_2^A\cap \Imm j_2^N
\subseteq H^2(\Sigma)$.
\end{dfn}

The point of this definition is that pairs $(N^{\prime},A^{\prime})\in\mathcal{M}(N)\times\mathcal{M}(A)$ satisfy the topological matching condition as well 
as the constraint on the AC rate of convergence to $C$.  Therefore, the set of gluing data near $(N,A)$ 
for which we can apply Theorem \ref{mainthm1} is a subset of 
$\mathcal{M}(N)\times\mathcal{M}(A)$.  By Theorem \ref{ACdefthm}, if we did not have the topological constraint on the deformations $A'$, then 
the moduli space of $A$ would be a smooth manifold with known dimension.  We shall see below that we can extend Theorem \ref{ACdefthm} in a straightforward way to 
show that $\mathcal{M}(A)$ is indeed a manifold.

We now conclude our moduli space definitions.

\begin{dfn}\label{moduliXdfn}
Let $\mathcal{M}(X)$ denote the moduli space of compact coassociative deformations of $X$, the coassociative 4-fold arising from gluing 
$N$ and $A$ as described above.  Theorems \ref{cptdefthm} and \ref{b2+thm} state that 
 $\mathcal{M}(X)$ is a smooth manifold of dimension $b^2_+(X)=b^2_+(\hat{N})+b^2_+(A)+\dim(\Imm j_2^A\cap\Imm j_2^N)$. 
\end{dfn}

By Theorem \ref{mainthm1}, we have a natural map from the gluing data into $\mathcal{M}(X)$; that is, given a pair $(N',A')\in\mathcal{M}(N)\times\mathcal{M}(A)$ 
we can define a deformation of $X$ by gluing $N'$ and $A'$.  Since we desingularize $N$ using $A$ to get $X$ it is natural to ask whether
 we can construct all compact coassociative 4-folds near $X$ via gluing; that is, whether the gluing map is a local diffeomorphism.  
In general this should not be possible, and the first thing to compare is the dimensions of $\mathcal{M}(N)$, $\mathcal{M}(A)$ and $\mathcal{M}(X)$.

We begin by recalling the description of $\mathcal{M}(N)$ from \cite{LotayCS}.

\begin{thm}\label{CSdefthm}  There exist finite-dimensional vector spaces of forms $\mathcal{I}(N)$ and $\mathcal{O}(N)$, an open neighbourhood 
$\hat{\mathcal{M}}(N)$ of 
$0$ in $\mathcal{I}(N)$ and a smooth map $\pi:\hat{\mathcal{M}}(N)\rightarrow\mathcal{O}(N)$ such that 
 $\mathcal{M}(N)$ near $N$ is locally homeomorphic to $\pi^{-1}(0)$ near $0$.  Moreover, the expected dimension 
of $\mathcal{M}(N)$ is 
\begin{equation*}
b^2_+(\hat{N})-
\!\!\!\!\!\!\sum_{\upsilon\in(-2,-1]}\!\!\!\!\!\!d_{\mathcal{D}}(\upsilon)-\ind(C),
\end{equation*}
where the stability index $\ind(C)$ is given in Definition \ref{stabdfn}. 
\end{thm}

\begin{proof}
The only difference between this result and the work in \cite{LotayCS} is that in \cite{LotayCS} 
just a lower bound was given for the expected dimension, but we can improve this using the work in this paper and in \cite{Lotaystab}. The work in \cite[$\S$7-$\S$8]{LotayCS} 
states that 
$$\dim\mathcal{I}(N)-\dim\mathcal{O}(N)=\dim\Ker(\d_++\d^*)_{\mu}-\dim\Ker(\d^*_++\d)_{\mu}+\dim\mathcal{H}^3_{\mu}+\dim\mathcal{C}$$
where $(\d_++\d^*)_{\mu}$ acts on $L^2_{4,\mu}(\Lambda^2_+T^*\hat{N}\oplus\Lambda^4T^*\hat{N})$ by $(\alpha,\beta)\mapsto \d\alpha+\d^*\beta$ 
as in \eq{lindefmap}, $(\d^*_++\d)_{\mu}$ acts on $L^2_{4,-3-\mu}(\Lambda^3T^*\hat{N})$ as $\gamma\mapsto (\d^*_+\gamma,\d\gamma)$, $\mathcal{H}^3_{\mu}$ is 
the space of closed and coclosed 3-forms in $L^2_{4,-3-\mu}(\Lambda^3T^*\hat{N})$ as in \eq{H3eq} and $\mathcal{C}$ is the orbit of the cone $C$ under the 
action of $\GG_2\times\R^7$.  

To understand this, first observe that if $(\alpha,\beta)\in\Ker(\d_++\d^*)_{\mu}$ then $\d\alpha+\d^*\beta=0$  
implies $\beta$ is harmonic.  Since $\beta$ tends to zero at the singular point (as $\mu>0$), the maximum principle forces $\beta=0$ and hence 
$\Ker(\d_++\d^*)_{\mu}$ is isomorphic to $\Ker(\d_+)_{\mu}$, the closed self-dual 2-forms in $L^2_{4,\upsilon}$.  Second, 
$\Ker(\d^*_++\d)_{\mu}$ is isomorphic to 
the cokernel of $(\d_++\d^*)_{\mu}$ and $\mathcal{O}(N)$ is a subspace of the cokernel which is contained in $\d(L^2_{4,\mu}(\Lambda^2T^*N))$.  
Hence $\mathcal{O}(N)$ is transverse to the cokernel of $(\alpha,\beta)\mapsto \d\alpha+\d^*\beta$ acting on 
$L^2_{4,\mu}(\Lambda^2T^*N\oplus\Lambda^4T^*N)$, and this cokernel is isomorphic to $\mathcal{H}^3_{\mu}$.  The fact that we have a moduli space $\mathcal{M}(N)$ 
is such that the singularity $z'$ of $N'$ is not required to be at the same point as the singularity of $N$ (which corresponds to allowing translations of $C$) 
and we are free to make different identifications for the tangent cone of $N'$ at $z'$ with $C$ (which amounts to choosing an element of $\GG_2$) means that 
the dimension of the obstruction space is reduced by $\dim\mathcal{C}$ as claimed.

Let $\mu_+\in(-2,0)$ be such that $[\mu_+,0)\cap\mathcal{D}=\emptyset$.  
Now, using Proposition \ref{kerOprop}, we know the dimension of the kernel of $\d$ on $L^2_{4,\mu_+}(\Lambda^2_+T^*\hat{N})$
and if $(\alpha,\beta)\in\Ker(\d_++\d^*)_{\mu_+}$ then $\beta$ is again harmonic and so must be constant (since the least negative growth rate that a harmonic
 4-form can have on a 4-manifold is $-2$).  We also have by Proposition \ref{cccsprop} and Lemma \ref{mu-lem} that
 $\Ker(\d^*_++\d)_{\mu_+}-\dim\mathcal{H}^3_{\mu_+}=\dim\mathcal{O}(N,\mu_+)$.  

As shown in \cite{LotayCS}, the index formula from \cite{LockhartMcOwen} in this 
situation states that
 \begin{align*}\dim\Ker(\d_++&\d^*)_{\mu}-\dim\Ker(\d^*_++\d)_{\mu}\\
&=\dim\Ker(\d_++\d^*)_{\mu_+}-\dim\Ker(\d^*_++\d)_{\mu_+}-\!\!\!\sum_{\upsilon\in[0,1]}\!\!\!d_{\mathcal{E}}(\upsilon)
\end{align*}
 where $d_{\mathcal{E}}(\upsilon)$ is the dimension of homogeneous forms of rate $\upsilon$ 
in $\Lambda^2_+T^*C\oplus\Lambda^4T^*C$ which are in the kernel of $\d_++\d^*$.  However, it is shown in the proof of \cite[Proposition 4.11]{Lotaystab} that 
any such homogeneous form on $C$ which is transverse to the closed self-dual 2-forms of rate $\upsilon$ corresponds to a form which lifts, in the sense 
we have used before, either to a constant 4-form on $\hat{N}$ or to a closed and coclosed 3-form on $\hat{N}$ of rate $-3-\upsilon$.  

Putting this information together, we see that 
\begin{align*}
\dim\Ker(\d_++\d^*)_{\mu}&-\dim\Ker(\d^*_++\d)_{\mu}+\dim\mathcal{H}^3_{\mu}\\
&=\dim\Ker(\d_+)_{\mu_+}-\dim\mathcal{O}(N,\mu_+)-\!\!\!\sum_{\upsilon\in[0,1]}
\!\!\!d_{\mathcal{D}}(\upsilon)\\
&=b^2_+(\hat{N})-
\!\!\!\!\sum_{\upsilon\in(-2,1]}\!\!\!\!d_{\mathcal{D}}(\upsilon).
\end{align*}
The fact that $\mathcal{O}(N)$ is decreased by $\dim\mathcal{C}$ and the definition of the stability index $\ind(C)$ in Definition \ref{stabdfn} gives the result.
\end{proof}

\begin{remark}
As discussed in \cite{Lotaystab}, it is possible to generalize the deformation theory of $N$ so that the cone at the singularity deforms in a family $\mathcal{C}$
 and the analogous result to Theorem \ref{CSdefthm} holds with the stability index replaced by the $\mathcal{C}$-stability index.
\end{remark}

We can easily calculate the expected difference in $\dim\mathcal{M}(X)$ and $\dim\mathcal{M}(N)$:
\begin{equation}\label{codimeq}
b^2_+(A)+\dim(\Imm j_2^A\cap\Imm j_2^N)+\!\!\!\!\!\!\sum_{\upsilon\in(-2,-1]}\!\!\!\!\!\!d_{\mathcal{D}}(\upsilon)+\ind(C).
\end{equation}
We deduce 
that the higher the stability index of $C$, the ``less likely'' a compact coassociative 4-fold is going to develop a conical 
singularity modelled on $C$.  More precisely, if $N$ arises as the limit of a family of compact coassociative 4-folds diffeomorphic to $X$, then 
we may view $\mathcal{M}(N)$ as contained in the compactification $\overline{\mathcal{M}(X)}$ 
of the moduli space $\mathcal{M}(X)$, for example by viewing $X$ as an integral current and taking $\overline{\mathcal{M}(X)}$ to be the closure of $\mathcal{M}(X)$ in the
 space of integral currents, which then contains $\mathcal{M}(N)$ as CS coassociative 4-folds are integral currents.  
(See \cite[$\S$2.1]{Lotaystab}, for example, for a discussion of coassociative integral currents.)  The expected codimension of 
$\mathcal{M}(N)$ in the compactified moduli space $\overline{\mathcal{M}(X)}$ is then given by \eq{codimeq} and hence will be larger if $\ind(C)$ is larger, which indicates 
that it should be more difficult to exhibit conical singularities developing when the stability index of the cone is higher.

We now use the deformation theory for $A$ from \cite{LotayAC} to describe $\mathcal{M}(A)$.

\begin{thm}\label{ACdefthm2}
The space $\mathcal{M}(A)$ is a smooth manifold near $A$ of dimension
$$b^2_+(A)+\dim(\Imm j_2^A\cap\Imm j_2^N)+
\!\!\!\!\!\!\sum_{\upsilon\in(-2,-\frac{1}{2})}\!\!\!\!\!\!d_{\mathcal{D}}(\upsilon).$$
\end{thm}

\begin{proof}
The moduli space $\hat{\mathcal{M}}(A)$ of AC coassociative deformations of $A$ with cone $C$ and rate $\lambda_0$ is a smooth manifold by 
Theorem \ref{ACdefthm}.  Moreover, the tangent space $T_{A}\hat{\mathcal{M}}(A)$ is isomorphic to the closed self-dual 2-forms on $A$ 
in $L^2_{4,\lambda_0}$. 

We may define a smooth map $\pi:\hat{\mathcal{M}}(A)\rightarrow \Imm j_2^A\subseteq H^2(\Sigma)$ by $\pi(A^{\prime})=j_2^{A^{\prime}}[\varphi_0]$.  
Then $\d\pi|_A:T_A\hat{\mathcal{M}}(A)\rightarrow\Imm j_2^A$ is surjective by the work in \cite{LotayAC}, as explained after Theorem 
\ref{ACdefthm}.  Thus $\pi$ is a submersion, so it follows that $\pi^{-1}(\Imm j_2^A\cap \Imm j_2^N)=\mathcal{M}(A)$ is a smooth manifold of the 
claimed dimension.
\end{proof}

Theorems \ref{CSdefthm} and \ref{ACdefthm2} show that if $\mathcal{M}(N)$ is smooth of the expected dimension, then 
$$\dim\mathcal{M}(N)+\dim\mathcal{M}(A)=\dim\mathcal{M}(X)-\left(\ind(C)-
\!\!\!\!\!\!\sum_{\upsilon\in(-1,-\frac{1}{2})}\!\!\!\!\!\!d_{\mathcal{D}}(\upsilon)\right).$$
We deduce from Definition \ref{stabdfn} that the quantity in brackets is non-negative and vanishes 
if and only if $C$ is stable.  We conclude that, unless $C$ is stable,  our gluing method can only at most generate a subset of $\mathcal{M}(X)$ near 
$X$. 

We therefore from now on restrict our attention to the situation where $C$ is stable, which corresponds, in some sense, to the most probable type of 
conical singularity to occur by \eq{codimeq}.  It follows from Theorem \ref{CSdefthm} that $\mathcal{M}(N)$ is smooth.  
  
Proposition \ref{mainthm3} and Theorem \ref{mainthm1} also imply that for any $(N^{\prime},A^{\prime})\in\mathcal{M}(N)\times\mathcal{M}(A)$ 
there exists $\tau(N^{\prime},A^{\prime})>0$ and smooth compact coassociative 4-folds $N^{\prime}(t)$ for $0<t<\tau(N^{\prime},A^{\prime})$ formed by 
gluing $N^{\prime}$ and $tA^{\prime}$ which converge to $N^{\prime}$ as $t\rightarrow 0$.  (We can make $\tau(N^{\prime},A^{\prime})$ canonical by taking the supremum again.)  
 Observe further that $A^{\prime}\in\mathcal{M}(A)$ implies that $tA^{\prime}\in\mathcal{M}(A)$ for all $t>0$ and we may choose $\tau(N^{\prime},A^{\prime})$ such 
that $t\tau(N^{\prime},tA^{\prime})=\tau(N^{\prime},A^{\prime})$.  We can thus make the following definition.

\begin{dfn}\label{gmapdfn} Assuming $C$ is stable and using the notation above, let
$$\mathcal{M}(N,A)=\big\{(N^{\prime},tA^{\prime})\in\mathcal{M}(N)\times\mathcal{M}(A)\,:\,t\in\big(0,\tau(N^{\prime},A^{\prime})\big)\big\},$$
which is the moduli space of ``matching pairs''.   
We can define a smooth map $G:\mathcal{M}(N,A)\rightarrow\mathcal{M}(X)$ by
 $G(N^{\prime},tA^{\prime})=N^{\prime}(t)$.
\end{dfn}  

Having defined our ``gluing map'' $G$, we can show Proposition \ref{mainthm4}.

\begin{prop}\label{gmapprop}
The map $G$ is a local diffeomorphism.
\end{prop}

\begin{proof}  
Consider $\d G|_{(N,A)}:T_N\mathcal{M}(N)\oplus T_A\mathcal{M}(A)\rightarrow T_X\mathcal{M}(X)$.  
Recall that 
$$T_X\mathcal{M}(X)\cong\{\alpha\in L^2_4(\Lambda^2_+T^*X)\,:\,\d\alpha=0\}$$ and that the proof of Theorem \ref{cptdefthm} implies that we can use these closed self-dual 2-forms to define natural coordinates on the moduli space $\mathcal{M}(X)$.

 A consequence of the work in \cite{LotayCS}, Theorem \ref{ACdefthm2} and the stability of $C$ is that, for $\delta>0$ such that $(-1,-1+\delta)\cap\mathcal{D}=\emptyset$, we have
\begin{gather*}
T_N\mathcal{M}(N)\cong\{\alpha\in L^2_{4,-1+\delta}(\Lambda^2_+T^*\hat{N})\,:\,\d\alpha=0\}\quad\text{and}\\
T_A\mathcal{M}(A)\cong\{\alpha\in L^2_{4,-1+\delta}(\Lambda^2_+T^*A)\,:\,\d\alpha=0,\,j_2^A[\alpha]\in\Imm j_2^N\}.
\end{gather*}
Moreover, we can use these spaces of closed self-dual 2-forms to define natural coordinates on $\mathcal{M}(N)$ and $\mathcal{M}(A)$ respectively.
  
In our construction of the approximation of the closed self-dual 2-forms on $X$ in Definition \ref{apKerdfn}, we used the same spaces of forms on $N$ and $A$ as above except with the 
weighted Sobolev space $L^2_{4,-2+\delta}$.  The crucial fact was that the topological condition that $j_2^A[\alpha]\in\Imm j_2^N$ enabled us to match closed 
self-dual 2-forms $\alpha$ on $A$ with decay of order $O(r^{-2})$ to closed self-dual 2-forms on $\hat{N}$ with the same decay, 
and thus effectively interpolate between them to construct 
our desired self-dual 2-form which is ``almost'' closed.  Since $C$ is stable, the analytic matching condition in Definition \ref{matchingdfn} is always satisfied, meaning in particular that any closed self-dual form on $A$ of order $O(r^{\upsilon})$ for $\upsilon\in(-2,-1]$ can be matched with a corresponding 
closed self-dual 2-form on $\hat{N}$.  

Moreover, we have from Theorems \ref{CSdefthm} and \ref{ACdefthm2} that
\begin{align*}
\dim\mathcal{M}(N)&=b^2_+(\hat{N})-
\!\!\!\!\!\!\sum_{\upsilon\in(-2,-1]}\!\!\!\!\!\!d_{\mathcal{D}}(\upsilon)\quad\text{and}\\
\dim\mathcal{M}(A)&=b^2_+(A)+\dim(\Imm j_2^A\cap\Imm j_2^N)+\!\!\!\!\!\!\sum_{\upsilon\in(-2,-1]}\!\!\!\!\!\!d_{\mathcal{D}}(\upsilon)
\end{align*}
since $d_{\mathcal{D}}(\upsilon)=0$ for $\upsilon\in(-1,-\frac{1}{2})$ by the stability of $C$. It therefore follows from Theorem \ref{b2+thm} that 
$$\dim\mathcal{M}(N)+\dim\mathcal{M}(A)=b^2_+(\hat{N})+b^2_+(A)+\dim(\Imm j_2^A\cap\Imm j_2^N)=\dim\mathcal{M}(X).$$

We conclude therefore, in the same way as for our approximate kernel in Definition \ref{apKerdfn}, that we may define a natural isomorphism 
between $T_N\mathcal{M}(N)\oplus T_A\mathcal{M}(A)$ and $T_X\mathcal{M}(X)$.  We may thus identify the product of the natural coordinates on $\mathcal{M}(N)$
and $\mathcal{M}(A)$ with the natural coordinates on $\mathcal{M}(X)$.   With this identification $\d G|_{(N,A)}$ becomes the identity map and hence $G$ 
is a local diffeomorphism.
\end{proof}  

Thus all compact coassociative 4-folds near $X$ arise via gluing in the stable case.  Moreover, 
Proposition \ref{gmapprop} suggests that all elements of $\mathcal{M}(X)$ ``sufficiently close'' to $\mathcal{M}(N)$, thought of as lying in the ``boundary'' of
compactified moduli space $\overline{\mathcal{M}(X)}$,  arise via the 
desingularization given by Theorem \ref{mainthm1}.  

\begin{remark}
It is also possible to extend our discussion of the stable case to where $C$ is $\mathcal{C}$-stable, as long as one knows that for every deformation $C^{\prime}$ of $C$ in $\mathcal{C}$ there is a corresponding deformation $A^{\prime}$ of $A$ which is AC to $C^{\prime}$.
\end{remark}

\subsection{Examples}

We now wish to discuss applications of our theory in examples.  We recall that to apply our results we need
\begin{itemize}
\item a coassociative 4-fold $N$, in an almost $\GG_2$ manifold $M$, with a conical singularity $z$ modelled on a cone $C\cong\R^+\times\Sigma$ and
\item a coassociative 4-fold $A\subseteq\R^7$ asymptotically conical with rate $\lambda<-\frac{1}{2}$ to $C$
\end{itemize} 
such that $A$ and $N$ satisfy the matching condition given in Definition \ref{matchingdfn}.  

A particular 
criterion for when this matching condition is satisfied is given 
in Proposition \ref{mainthm3}, namely that the topological matching condition holds (see Definition \ref{topmatchingdfn}) and the cone $C$ is $\mathcal{C}$-stable
 in the sense of Definition \ref{stabdfn}.  Recall that stability is related to the exceptional rates $\mathcal{D}$ given in Definition \ref{Ddfn}.  
Moreover, in the case when $\lambda\leq -2$ the matching condition is equivalent to the topological matching condition.

We now give examples of situations where we can apply our results and begin with a degenerate case.

\begin{ex}
Suppose we make the perverse choice that $N$ is smooth and $z$ is any point.  Then $C=\R^4$ and $\Sigma\cong\mathcal{S}^3$ so
 $b^1(\Sigma)=0$ and the topological matching condition will hold trivially.  
We can take $A=\R^4$ which is obviously AC with any negative rate, so the matching condition is satisfied.  Since $ b^2_+(A)=0$ and $b^2_+(\hat{N})=b^2_+(N)$, 
applying Corollary \ref{mainthm2} gives that there is a $b^2_+(N)$-dimensional deformation family of coassociative ``smoothings'' of $N$.  This 
 corresponds to the fact that our gluing construction will just give back $N$ in this case.
\end{ex}

Since $\R^4$ is stable by \cite[Corollary 5.7]{Lotaystab}, Proposition \ref{mainthm3} shows that 
Theorem \ref{mainthm1} would apply to gluing in any asymptotically planar but non-planar $A$ 
with rate $\lambda<-\frac{1}{2}$ into smooth $N$.  However, we can show that no such $A$ exists.

\begin{prop}\label{planeprop}
If a coassociative 4-fold $A$ in $\R^7$ is AC with rate $\lambda<0$ to a coassociative 4-plane $C$, then $A=C$.
\end{prop}

\begin{proof}
Suppose, for a contradiction, that there is a least 
choice of $\lambda$ such that $\lambda\geq -2$.  Since $\mathcal{D}\cap(-3,0)=\emptyset$ 
by \cite[Corollary 5.7]{Lotaystab}, it follows from Proposition \ref{alphaAprop} that $A$ can be written as the graph of a self-dual 2-form 
on $C$ which has decay rate $O(r^{\text{max}\{2\lambda-1,\lambda_-\}})$ where $\lambda_-<-2$.  Thus $A$ is also AC with rate
 $\text{max}\{2\lambda-1,\lambda_-\}<\lambda$, which gives our required contradiction.
 
Since $\lambda<-2$ and the stabilizer of $C$ in $\GG_2$ is $\SO(4)$, we can apply \cite[Proposition 9.1]{LotayAC} and deduce that $A$ is also $\SO(4)$-invariant.  
However, this $\SO(4)$ action decomposes $\R^7=C\oplus C^{\perp}=\R^4\oplus\R^3$, so if $A$ is $\SO(4)$-invariant it must equal $C$.
\end{proof}

We now relate our results to the work in \cite{Lotaydesing} and discuss natural extensions. 

\begin{ex}\label{unobstructex}  In the desingularization theory in \cite{Lotaydesing}, we assumed that $b^1(\Sigma)=0$ and $\lambda<-2$.  In this case, the matching condition is equivalent to the topological matching condition (as $\lambda<-2$), but this is trivially satisfied since 
$b^1(\Sigma)=0$.  Applying Theorem \ref{mainthm1} gives nothing but the main result in \cite{Lotaydesing} and we have a deformation family of coassociative
 smoothings of $N$ of dimension $b^2_+(A)+b^2_+(N)$ by Corollary \ref{mainthm2}.   

In fact, if we drop the assumption that $b^1(\Sigma)=0$, our topological matching condition is still met since $\lambda<-2$ means that $\alpha_A^0$, given in 
Proposition \ref{alphaAprop}, vanishes.  Hence Theorem \ref{mainthm1} applies and clearly extends the work in \cite{Lotaydesing}.
\end{ex}

\begin{ex}\label{obstructex}
If we assume that $b^1(\Sigma)\neq 0$ (so $-2\in\mathcal{D}$) and $\lambda\leq -2$, our matching condition is potentially non-trivial and equivalent to the 
topological matching condition.  If we then assume further that the topological criterion holds, we may apply Theorem \ref{mainthm1} and deduce that we 
can smooth $N$ using $A$ via gluing.  This situation is directly analogous to the material in \cite{JoyceCS4} on desingularization of special Lagrangian conical singularities in what is described as ``the obstructed case''.
\end{ex}

\begin{remark} Example \ref{obstructex} shows that the work in this article provides an extension of the desingularization theory in
 the coassociative world which has no current analogue in special Lagrangian geometry, but which should surely follow by adapting the ideas presented here.
\end{remark}

Perhaps the best known example of a non-trivial coassociative cone is the Lawson--Osserman $\SU(2)$-invariant cone (see \cite{LO} or \cite[Example 4.2]{Lotaystab}, for example), originally exhibited because it gives an 
example of an area-minimizing Lipschitz submanifold which is not smooth.  This cone gives a natural model for a coassociative conical singularity.

\begin{ex}\label{HLex}
Suppose $C$ is the Lawson--Osserman cone.  Then $\Sigma\cong\mathcal{S}^3$, $\mathcal{D}\cap(-2,0)=\{-\frac{3}{2}\}$ 
and $d_{\mathcal{D}}(-\frac{3}{2})=1$ by \cite[Corollary 5.8]{Lotaystab}.    Moreover, we have a dilation family of AC smoothings $A$ of $C$
with rate $-\frac{3}{2}$ which have $b^2_+(A)=0$ (see \cite[Theorem IV.3.2]{HL} or \cite[Proposition 9.3]{LotayAC}, for example).  

Therefore the topological matching condition is trivially satisfied, so the matching condition is equivalent to
 the dilation deformation of $A$ extending to an infinitesimal deformation of $N$, which is obviously essential for the 
smoothings of $N$ to exist via gluing.  The stability of $C$ \cite[Corollary 5.8]{Lotaystab} implies that this always occurs 
by Proposition \ref{mainthm3}.

Applying Corollary \ref{mainthm2} gives a family of coassociative smoothings $X$ of $N$ of dimension $b^2_+(\hat{N})$.  Notice that the stability of $C$ means 
that $N$ has a smooth moduli space of deformations as a CS coassociative 4-fold of dimension $b^2_+(\hat{N})-d_{\mathcal{D}}(-\frac{3}{2})=b^2_+(\hat{N})-1$ by Theorem \ref{CSdefthm}.  
Hence, $\dim\mathcal{M}(N)=\dim\mathcal{M}(X)-1$, recalling the notation of Definitions \ref{moduliNdfn}-\ref{moduliXdfn}.  Moreover, since $\mathcal{M}(A)\cong\R$, 
the gluing map $G$ given in Definition \ref{gmapdfn} acts between $\mathcal{M}(N)\times(0,\tau)$ and $\mathcal{M}(X)$, and
Proposition \ref{mainthm4} shows that $G$ is a local diffeomorphism.   Thus, all nearby compact coassociative deformations of $X$ arise 
in one-parameter families which degenerate to elements of $\mathcal{M}(N)$.
\end{ex}

\begin{remarks}
In the situation of Example \ref{HLex}, it is natural to speculate whether we can view $\mathcal{M}(N)$ as (the top stratum of) the boundary of 
the compactified moduli space $\overline{\mathcal{M}(X)}$; that is, 
whether we can identify $\mathcal{M}(N)\times[0,\tau)$ with a neighbourhood of $N$ in $\overline{\mathcal{M}(X)}$.  
 The hope would be to prove that every coassociative integral current 
close to $N$ is either a CS deformation of $N$ or else arises via gluing.  This would be the coassociative analogue of \cite{Imagi}. 
\end{remarks}

Using Proposition \ref{alphaAprop} and the classification of $\SU(2)$-invariant coassociative 4-folds in a similar manner to the proof of 
Proposition \ref{planeprop},  one may deduce that any coassociative 4-fold that is AC with rate $\lambda<0$ to the  Lawson--Osserman cone $C$ must have 
 $\lambda=-\frac{3}{2}$.  As far as the author is aware, it
 is  an open question whether the $\SU(2)$-invariant coassociative 4-folds given in Example \ref{HLex} are the unique AC coassociative 
 4-folds asymptotic to $C$ -- they are certainly locally unique by the work in \cite{LotayAC}.  If they are unique, then there is a unique possible method of desingularizing conical singularities
 modelled on $C$ via gluing, given by Theorem \ref{mainthm1}.

\begin{ex}\label{Foxex}
Fox \cite[Example 9.2]{Fox} generalized the Lawson--Osserman cone to define a cone $C(\Gamma)$ given 
any (non-totally geodesic) null-torsion pseudoholomorphic curve $\Gamma$ in $\mathcal{S}^6$.  (The curves $\Gamma$ were
 defined by Bryant \cite[$\S$4]{Bryant} -- see, for example, \cite[$\S$3.2]{LotayLag} for a definition.  The Lawson--Osserman example corresponds to
 choosing $\Gamma$ to be a totally geodesic $\mathcal{S}^2$.)   

Moreover, $C(\Gamma)$ admits a dilation family of smoothings $A(\Gamma)$ which converge with rate $-\frac{3}{2}$ to (possibly a finite cover of) $C(\Gamma)$
 (c.f.~\cite[Theorem 9.3]{Fox}).  
If $C=C(\Gamma)$ and $A=A(\Gamma)$ is AC, then our theory applies whenever the matching condition holds, which is a non-trivial constraint.  
\end{ex}

\begin{remarks}
Taking the curve $\Gamma$ in Example \ref{Foxex} to be the constant curvature $\frac{1}{6}$ null-torsion pseudoholomorphic $\mathcal{S}^2$ in 
$\mathcal{S}^6$ (the so-called Bor\r{u}vka sphere) leads to a cone with constant curvature $\frac{1}{16}$ link $\SO(3)/\AAA_4$  
(c.f.~\cite[$\S$6.3]{LotayLag}).  However, one can check that the end of $A(\Gamma)$ is diffeomorphic to $\R^+\times\SO(3)/\Z_3$ and so $A(\Gamma)$ is \emph{not} AC in our sense, 
and thus cannot be used in our gluing procedure.
\end{remarks}

We now conclude with examples arising from complex geometry.

\begin{ex}\label{cxex}
If $C$ is complex, \cite[Theorem 6.5]{Lotaystab} shows that $(-2,0)\cap\mathcal{D}=\{-1\}$ and $d_{\mathcal{D}}(-1)$ can be determined by
 the degree of the holomorphic curve in $\C\P^2$ which is the complex link of $C$.  Thus for any 
AC smoothing of $C$ either $\lambda\leq -2$ or $\lambda=-1$.  We can therefore apply our theory whenever the matching condition holds, which is 
non-trivial to check in general.  A case of particular interest is discussed in the next example.
\end{ex}

\begin{remark}
The same situation as Example \ref{cxex} holds if the link $\Sigma$ of $C$ is a tube of radius $\frac{\pi}{2}$ in the second normal bundle
 of a null-torsion pseudoholomorphic curve in $\mathcal{S}^6$ (see \cite[Example 6.12]{LotayLag} for a description of such $\Sigma$). 
\end{remark}

\begin{ex}\label{mainex}
In \cite[Theorem 1.3]{Lotaystab} the author constructed CS coassociative 4-folds $N$ in holonomy $\GG_2$ manifolds which are diffeomorphic to $K3$ surfaces, so
$b^2_+(\hat{N})=3$.  The cone $C$ at the singularity is complex with $\Sigma\cong\R\P^3$.  

There is a 2-parameter family of AC smoothings $A$ of the cone at the
 singularity which have rate $-1$ and $b^2_+(A)=0$.  
Moreover, $d_{\mathcal{D}}(-1)=2$, corresponding to the choices for $A$.  
Thus, since the topological matching condition is trivially satisfied, the matching condition holds if and only if 
 the 2-parameter family of deformations of $A$ extend to infinitesimal deformations of $N$.  

By \cite[Corollary 6.11]{Lotaystab}, $C$ is $\mathcal{C}$-stable for some natural choice of family of cones $\mathcal{C}$, so Proposition \ref{mainthm3} 
implies that we may apply our theory and Corollary \ref{mainthm2} gives us a 3-dimensional family of smoothings of $N$.   In fact, since the coassociative
 4-folds arise initially in a fibration (see \cite{Kovalev} and \cite[$\S$7]{Lotaystab}), we see that this family of smoothings is maximal.  Notice also that the $\mathcal{C}$-stability of $C$ implies that 
 $N$ has a smooth moduli space of CS deformations of dimension $b^2_+(\hat{N})-d_{\mathcal{D}}(-1)=1$.  

Thus, in the notation of Definitions \ref{moduliNdfn} and \ref{moduliXdfn}, $\dim\mathcal{M}(N)=\dim\mathcal{M}(X)-2$, which agrees with the fact that singular fibres in the 
 fibration arise in $\mathcal{S}^1$-families and are thus codimension two in the space of smooth fibres.  Moreover, for every deformation of $C$ in $\mathcal{C}$ 
there is a corresponding deformation of $A$ which is AC to the deformed cone.  
Therefore, as remarked after Proposition \ref{gmapprop}, the gluing map defines a local diffeomorphism 
$G:\mathcal{M}(N)\times B(0;\tau)\rightarrow \mathcal{M}(X)$ where $B(0;\tau)\subseteq\R^2$.  Thus every smooth fibre near $X$ arises via gluing and comes in a
 2-parameter family which degenerates to a singular fibre in $\mathcal{M}(N)$.
\end{ex}

\begin{ack}
The author is grateful to Dominic Joyce, Alexei Kovalev, Lorenzo Mazzieri and Tommaso Pacini for useful conversations and comments.  
\end{ack}

\end{document}